\documentclass[12pt]{amsart}

\usepackage{amsmath}
\usepackage{amsfonts}
\usepackage{amsthm}
\usepackage{amssymb}
\usepackage{graphicx}
\usepackage{tikz}
\usepackage{latexsym}
\usepackage{mathrsfs}
\usepackage[applemac]{inputenc}
\usetikzlibrary{positioning}
\usetikzlibrary{matrix}
\usetikzlibrary{arrows}
\usetikzlibrary{decorations.markings,positioning}
\usepackage{version}
\usepackage{bbm} 
\usepackage[square, numbers]{natbib}
\usepackage{xcolor}
\usepackage[all,import]{xy}

\theoremstyle{definition}
\newtheorem{thm}[equation]{Theorem}
\newtheorem{theorem}[equation]{Theorem}
\newtheorem*{theorem*}{Theorem}
\newtheorem*{com*}{Comments}

\newtheorem{corollary}[equation]{Corollary}

\newtheorem{lemma}[equation]{Lemma}

\newtheorem{proposition}[equation]{Proposition}
\newtheorem*{ack*}{Acknowlegements}
      \newtheorem{defi}[equation]{Definition}

      \newtheorem{remark}[equation]{Remark}

      \newtheorem{example}[equation]{Example}

      \numberwithin{equation}{section}

\newcommand{\nc}{\newcommand}
\newcommand{\adj}[4]{#1\negmedspace: #2\rightleftarrows #3:\negmedspace #4}
\nc{\DMO}{\DeclareMathOperator}
\newcommand{\bigslant}[2]{{\raisebox{.2em}{$#1$}\left/\raisebox{-.2em}{$#2$}\right.}}	
\nc{\Cob}{\mathsf{Cob}}
\nc{\cob}{\mathsf{Cob}}
\nc{\lag}{\mathsf{Lag}}
\nc{\fun}{\mathsf{Fun}}
\nc{\cat}{\mathsf{Cat}}
\nc{\vect}{\mathsf{Vect}}
\nc{\sets}{\mathsf{Sets}}
\nc{\symp}{\mathsf{Symp}}
\nc{\corr}{\mathsf{Corr}}
\nc{\fuk}{\mathsf{Fukaya}}
\nc{\chain}{\mathsf{Chain}}
\nc{\coder}{\mathsf{Coder}}
\nc{\ssets}{\mathsf{sSets}}
\nc{\cmpct}{\mathsf{cmpct}}
\nc{\grmod}{\mathsf{GrMod}}
\nc{\dbcoh}{D^b\mathsf{Coh}}
\nc{\fukaya}{\mathsf{Fukaya}}
\nc{\spaces}{\mathsf{Spaces}}
\nc{\corres}{\mathsf{Corres}}
\nc{\inftycat}{\infty\mathsf{Cat}}

\DMO{\conf}{Conf}
\DMO{\chains}{Chains}
\DMO{\cochains}{Cochains}
\DMO{\cone}{Cone}
\DMO{\ran}{Ran}
\DMO{\leg}{Leg}
\DMO{\cube}{Cube}
\DMO{\floer}{Floer}
\DMO{\holomaps}{Holomaps}
\DMO{\maps}{Maps}
\DMO{\Decomp}{Decomp}
\DMO{\decomp}{Decomp}
\DMO{\yoneda}{Yoneda}
\DMO{\strict}{strict}
\DMO{\comp}{Comp}
\DMO{\crit}{Crit}
\DMO{\test}{{test}}
\DMO{\sign}{sign}
\DMO{\topp}{top}
\DMO{\indx}{Index}
\DMO{\Break}{Break} 
\DMO{\zero}{zero} 
\DMO{\ob}{Ob}
\DMO{\gr}{Gr} 
\DMO{\Gr}{Gr} 
\DMO{\cl}{Cl} 
\DMO{\grlag}{GrLag}
\DMO{\Pin}{Pin}
\DMO{\Graph}{Graph}
\DMO{\pin}{Pin}
\DMO{\gap}{Gap}
\DMO{\Ex}{Ex}
\DMO{\id}{id}
\DMO{\End}{End}
\DMO{\sym}{Sym} 
\DMO{\aut}{Aut}
\DMO{\DK}{DK} 
\DMO{\poly}{poly} 
\DMO{\diff}{Diff} 
\DMO{\dist}{dist} 
\DMO{\coker}{coker} 
\nc{\kernel}{\ker} 
\DMO{\sspan}{span}
\DMO{\hocolim}{hocolim}	
\DMO{\holim}{holim}
\DMO{\sk}{sk}

\nc{\xto}{\xrightarrow}
\nc{\xra}{\xto}
\nc{\tensor}{\otimes}
\nc{\del}{\partial}
\nc{\delbar}{\overline{\del}}
\nc{\dd}{\diamond}
\nc{\tri}{\triangle}
\nc{\bb}{\Box}
\nc{\into}{\hookrightarrow}
\nc{\contains}{\supset}

\nc{\trbar}{\overline{T^*\RR}}
\nc{\tsa}{Ts\cA}
\nc{\tsb}{Ts\cB}
\nc{\Ainf}{\mathcal{A}_{\infty}}
\nc{\Cinf}{\mathcal{C}_{\infty}}
\nc{\vece}{ {\vec \epsilon}}	
\nc{\vecd}{ {\vec \delta}}
\nc{\vt}{ {\vec t}}
\nc{\vx}{ {\vec x}}
\nc{\vs}{ {\vec s}}

\DMO{\op}{op}

\nc{\eqn}{\begin{equation}}
\nc{\eqnd}{\end{equation}}

\nc{\hiro}{\textcolor{blue}}
\nc{\gio}{\textcolor{green}}

\def\cA{\mathcal A}\def\cB{\mathcal B}

\def\RR{\mathbb R}

\begin{document}

\setcounter{section}{-1}
\setcounter{tocdepth}{1}

\title{$\Ainf$-functors and homotopy theory of dg-categories}

\author{Giovanni Faonte}
\address[Giovanni Faonte]{Department of Mathematics\\
        Yale University\\
       10 Hillhouse Avenue, New Haven CT 06520.}
\email[G.~Faonte]{giovanni.faonte@yale.edu}


\begin{abstract}
In this paper we prove that T\"{o}en's derived enrichment of the model category of dg-categories defined by Tabuada, is computed by the dg-category of $\Ainf$-functors. This approach was suggested by Kontsevich. We further put this construction into the framework of $(\infty,2)$-categories. Namely, we enhance the categories $dgCat$ and $\Ainf Cat$, of dg and $\Ainf$-categories, to $(\infty,2)$-categories using the nerve construction of \cite{Fao} and the $\Ainf$-formalism. We prove that the $(\infty,1)$-truncation of to the $(\infty,2)$-category of dg-categories is a model for the simplicial localization at the model structure of Tabuada. As an application, we prove that the homotopy groups of the mapping space of endomorphisms at the identity functor in the $(\infty,2)$-category of $\Ainf$-categories compute the Hochschild cohomology.
\end{abstract}

\maketitle

\tableofcontents

\section{Introduction}

Differential graded categories and $\Ainf$-categories have been subject of study in non commutative geometry and symplectic geometry. Remarkable work has been done by Drinfeld \cite{Dr} and Keller \cite{Ke3} defining dg-quotients of dg-categories. Related to this notion, is the existence of a model category $(dgCat,Tab)$ of dg-categories defined by Tabuada in \cite{Tab}. It is known that the category of dg-categories has a structure of closed symmetric monoidal category 
\medskip
\begin{equation}\label{eq1}
Hom_{dgCat}(C\otimes D,E)\cong Hom_{dgCat}(C,dgFun^{\bullet}(D,E))
\end{equation}
where $-\otimes -$ is the tensor product of dg-categories and $dgFun^{\bullet}(-,-)$ is the dg-category of dg-functors. However, as pointed out by T\"{o}en \cite{To}, the tensor product is not compatible with the model category defined in \cite{Tab}, in the sense that it does not preserve cofibrant objects. In particular the adjunction (\ref{eq1}) is not a Quillen adjunction of two variables \cite{Hov}. The tensor product can still be left derived defining a monoidal structure on the homotopy category of the model category $(dgCat,Tab)$
\medskip
\[
-\otimes^{\mathbb{L}} -: Ho(dgCat) \times Ho(dgCat_) \to Ho(dgCat)
\]
The result of T\"{o}en tells that this monoidal structure is closed.
\medskip
\begin{thm}\cite{To}
The monoidal category $(Ho(dgCat),\otimes^\mathbb{L})$ is closed. Namely, given dg-categories $D,E$, there exists an object of 
\medskip
\[
\mathbb{R}Hom(D,E)\in Ho(dgCat)
\] 
and natural isomorphisms 
\medskip
\[
Hom_{Ho(dgCat)}(C\otimes^\mathbb{L} D,E)\cong Hom_{Ho(dgCat)}(C,\mathbb{R}Hom(D,E))
\] 
\end{thm}
T\"{o}en's description of the derived enrichment is rather implicit: it involves a certain dg-category of right quasi-representable dg-bimodules. Many authors \cite{Dr}, \cite{Ke1}, \cite{To} refer to a result of Kontsevich stating that the derived enrichment is given by a more explicit dg-category, $\mathcal{A}_{\infty}(D,E)$, whose objects are $\Ainf$-functors from $D$ to $E$. However, no proof of this fact can be found in the literature. The first result of this paper is a proof of this statement.
\medskip
\begin{thm}\label{th1}
Given dg-categories $D,E$, there exists natural isomorphisms in $Ho(dgCat)$ 
\[
\mathbb{R}Hom(D,E)\xrightarrow{\simeq}  \mathcal{A}_{\infty}(D,E)
\]
\end{thm}
Next, we develop and interpret this result in terms of $(\infty,2)$-categories. The second result of this paper is the definition of two $(\infty,2)$-categories: the first, $\Ainf Cat_{(\infty,2)}$, has as objects $\Ainf$-categories and $(\infty,1)$-categories of morphisms given by
\medskip
\[
\Ainf Cat_{(\infty,2)}(A,B)=N_{\Ainf}(\mathcal{A}_{\infty}(A,B))
\]

Here $\mathcal{A}_{\infty}(A,B)$ is the $\Ainf$-category of unital $\Ainf$-functors as defined in \cite{Ly}, \cite{LH}, \cite{Sei} and $N_{\Ainf}$ is the $\Ainf$-nerve introduced in \cite{Fao}. The bar construction of $\mathcal{A}_{\infty}(A,B)$ provides an enrichment of the category of $\Ainf$-categories over the the monoidal category of dg-cocategories which is used to define the structure of $(\infty,2)$-category on $\Ainf Cat_{(\infty,2)}$. 

The second $(\infty,2)$-category we describe, $dgCat_{(\infty,2)}$, is the full $(\infty,2)$-subcategory of $\Ainf Cat_{(\infty,2)}$ whose objects are dg-categories. In this case, for dg-categories $C$ and $D$, the $\Ainf$-category $\mathcal{A}_{\infty}(C,D)$ is the dg-category of theorem \ref{th1}. Moreover we have that
\medskip
\[
dgCat_{(\infty,2)}(C,D)=N_{\Ainf}(\mathcal{A}_{\infty}(C,D))=N_{dg}(\mathcal{A}_{\infty}(C,D))
\]
where $N_{dg}$ is the dg-nerve of Lurie \cite{LHA}.

Our third result says that the associated $(\infty,1)$-category to $dgCat_{(\infty,2)}$ is a model for the simplicial localization of $(dgCat,Tab)$. Recall that Dwyer-Kan localization \cite{DK} associates to the model category $(dgCat,Tab)$ a simplicial category whose homotopy category is equivalent to the localization of $dgCat$ at the class of weak-equivalences of the Tabuada model structure. However, this simplicial category is not always the correct construction to consider in the context of $(\infty,1)$-categories. The reason is that the simplicial sets of morphisms in the Dwyer-Kan localization are not, in general, Kan fibrant simplicial sets. Models for the $(\infty,1)$-category associated to a model category exists when the model category is simplicial \cite{LHT}, which is not the case of $(dgCat,Tab)$. In general, a construction of an $(\infty,1)$-category associated to a category with a class of weak-equivalences can be defined \cite{LHA} but it is, in practice, difficult to manage for concrete applications. To an $(\infty,2)$ category $X$, we can associate an $(\infty,1)$-category, $X^{\circ}$ obtained by taking the maximal Kan complex contained in each of the $(\infty,1)$-category of morphisms in $X$. This procedure can be understood as a groupidification of $(\infty,1)$-category of morphisms. We prove that the $(\infty,1)$-category
\medskip
\[
dgCat_{(\infty,1)}= dgCat_{(\infty,2)}^{\circ}
\]
is a model for the $(\infty,1)$-category associated to the model category $(dgCat,Tab)$ in the following sense.
\medskip
\begin{thm}
Given dg-categories $D$ and $E$, there exists natural weak homotopy equivalences of simplicial sets
\medskip
\[
Map_{L_{Tab}(dgCat)}(C,D) \xrightarrow{} Map_{dgCat_{(\infty,1)}}(C,D) 
\]
where $Map_{L_{Tab}(dgCat)}(C,D)$ is the mapping space in the Dwyer-Kan localization of $(dgCat,Tab)$. 
\end{thm}

In the last section, as an application, we show how the Hochschild cohomology arises naturally from the $(\infty,2)$-categories $dgCat_{(\infty,2)}$ and $\Ainf Cat_{(\infty,2)}$. For a dg-category $D$, the Hochschild complex is defined as 
\medskip
\[
\mathbb{HH}(D,D)=\mathbb{R}Hom^{\bullet}_{Mod(D\otimes D^{op})}(D,D)
\] 
and its cohomology
\medskip
\[
HH^{i}(D,D)=H^{i}(\mathbb{HH}(D,D))
\] 
is the Hochschild cohomology. The approach of T\"{o}en to the computation of the derived enrichment of $(dgCat,Tab)$ via the dg-category of right quasi-representable dg-bimodules, provides an identification of the Hochschild complex of $D$ with the complex of endomorphism of $D$, seen as a dg-bimodule in the obvious way, in the dg-category $\mathbb{R}Hom(D,D)$. However, the Hochschild complex can be explicitly computed choosing a particular resolution \cite{FMT}, \cite{GeJo}, \cite{Ke2}, obtaining
\medskip
\[
\mathbb{HH}(D,D)\simeq Hom^{\bullet}_{\Ainf (D,D)}(Id_D,Id_D)
\]
This approach extends to $\Ainf$-categories, for which the Hochschild complex is given by
\medskip
\[
\mathbb{HH}(A,A)\simeq Hom^{\bullet}_{\Ainf (A,A)}(Id_A,Id_A)
\]
We remark that the approach via derived functor fails in the $\Ainf$-setting because of the lack of a model structure on the category of $\Ainf$-bimodules. We prove the following theorem, which generalise the computation of the Hochschild cohomology for dg-categories of \cite{To} to $\Ainf$-categories.
\medskip
\begin{thm}
For any $\Ainf$-category $A$, $i\ge 0$, we have
\medskip
\[
\pi_i(Map_{N_{\Ainf}(\Ainf (A,A))}(Id_A,Id_A))=HH^{-i}(A,A)
\]
\end{thm}
\medskip
\begin{com*}
It is well known \cite{GeJo} that the bar construction of the Hochschild complex of an $\Ainf$-category (or dg) carries a structure of a dg-bialgebra, whose operations induce cup product and Gerstenhaber bracket on the Hochschild cohomology. The bialgbera structure is given \cite{Ke2} by considering the complex
\medskip
\[
B(Hom^{\bullet}_{\Ainf (A,A)}(Id_A,Id_A))
\]
as the endomorphims coalgebra of $Id_A$ in the enriched category of $\Ainf$-categories. Those ideas relate to the question of what dg-categories and $\Ainf$-categories form \cite{Tam}. Any possible answer should recover, in some form, the bialgebra structure of the Hoschschild complex. In this sense, the $(\infty,2)$-categories $dgCat_{(\infty,2)}$ and $\Ainf Cat_{(\infty,2)}$ seem to be candidates for this purpose. For a dg (or $\Ainf$)-category $D$ the mapping space of endomorphims at $Id_D$
\medskip
\begin{equation}\label{eq27}
End_{dgCat_{(\infty,2)}}(Id_D)=Map_{N_{dg}(\Ainf (D,D))}(Id_D,Id_D)
\end{equation}
comes equipped with two maps, one given by the $(\infty,2)$-category structure of $dgCat_{(\infty,2)}$ and the other from being the simplcial set of endomorphims in the $(\infty,1)$-category $N_{dg}(\Ainf (D,D)$ which, by construction, are related to the endomorphism coalgebra (\ref{eq27}), and its homotopy groups compute the Hochschild cohomology. 
\end{com*}
\medskip
\begin{ack*}
The author would like to thank his doctoral advisor Mikhail Kapranov and Hiro-Lee Tanaka for useful discussions through the realization of this paper. This work was partially supported by World Premier International Research Initiative (WPI), MEXT, Japan.
\end{ack*}

\newpage

\section{Homotopy theory of dg-categories and derived enrichment.}

In this section we recall preliminary results about the homotopy theory of dg-categories and the construction of the $\Ainf$-category of $\Ainf$-functors $\Ainf(A,B)$ between two given $\Ainf$-categories. We then prove theorem \ref{th1}. 

	\subsection{The Tabuada model structure on dg-categories and T\"{o}en's derived enrichment.}
From now on we fix a field $\mathbb{K}$ of characteristic $0$. Tabuada in \cite{Tab} defines a model structures on the category $dgCat$ (see appendix \ref{APPA}) that we recall.
\medskip
\begin{proposition}
There exists a cofibrantly generated model category structure on $dgCat$, denoted by $(dgCat,Tab)$, for which weak-equivalences  are dg-functors $f\in Hom_{dgCat}(D,E)$ such that for every pair of objects $x,y$ in $D$, the induced map of cochain complexes
\medskip
\[
f_{x,y}:Hom^{\bullet}_{D}(x,y)\to Hom^{\bullet}_{E}(f(x),f(y))
\]
is a quasi-isomorphism, and the induced functor 
\medskip
\[
H^0(f): H^0(D)\to H^0(E)
\]
is essentially surjective. Fibrations are dg-functors such that for every pair of objects $x,y$ in $D$, the induced map of cochain complexes
\medskip
\[
f_{x,y}:Hom^{\bullet}_{D}(x,y)\to Hom^{\bullet}_{E}(f(x),f(y))
\]
is an epimorphism and, for every object $x$ in $D$ and isomorphism $[v]\in Hom_{H^0(E)}(f(x),z)$, there exists an isomorphism $[u]\in Hom^{\bullet}_{H^0(D)}(x,y)$ such that $f(u)=v$. Cofibrations are dg-functors satisfying left lifting property with respect to trivial fibrations. In this model category every dg-category is fibrant. 
\end{proposition}
\medskip
\begin{remark}
The category $dgCat$ carries a closed symmetric monoidal structure
\medskip
\begin{equation}
\begin{gathered}
-\otimes-:dgCat \times dgCat \to dgCat \\
dgFun^{\bullet}(-,-): dgCat^{op}\times dgCat \to dgCat
 \end{gathered}
 \end{equation} \label{eq2}
where $-\otimes-$ is the tensor product of dg-categories and $dgFun^{\bullet}(D,E)$ is the dg-category of dg-functors (see appendix \ref{APPA}). The adjunction (\ref{eq2}) do not define a Quillen adjunction of two variables \cite{Hov} but the tensor product can be left derived \cite{To} 
\medskip
\[
-\otimes^\mathbb{L}-: Ho(dgCat) \times Ho(dgCat) \to Ho(dgCat)
\]
by the formula
\medskip
\[
D \otimes^\mathbb{L}E=Q(D) \otimes Q(E)
\]
defining a monoidal structure on $Ho(dgCat)$ which is closed by the result of T\"{o}en.
\end{remark}
\medskip
\begin{thm}\cite{To}
The monoidal category $(Ho(dgCat),\otimes^\mathbb{L})$ is closed. Namely, given dg-categories $D,E$, there exists an object of 
\medskip
\[
\mathbb{R}Hom(D,E)\in Ho(dgCat)
\] 
and natural isomorphisms 
\medskip
\[
Hom_{Ho(dgCat)}(C\otimes^\mathbb{L} D,E)\cong Hom_{Ho(dgCat)}(C,\mathbb{R}Hom(D,E))
\] 
The dg-category $\mathbb{R}Hom(D,E)$ is equivalent to the full sub dg-category of right quasi-representable $Q(D)$-$E$ dg-bimodules 
\medskip
\[
Int((Mod^{\bullet}(Q(D),E)^{rqr})
\]
whose objects are fibrant and cofibrant dg-modules (see appendix \ref{APPB}).
\end{thm}

\subsection{Derived enrichment via the dg-category of $\Ainf$-functors.}
$\Ainf$-categories and $\Ainf$-functors have been introduced by Fukaya \cite{Fuk} as a generalization of the notion of $\Ainf$-algebra due to Stasheff \cite{Sta}. They are an not strictly associative version of dg-categories. For a precise definition we refer to appendix \ref{APPA}. There is a construction, due originally to Kontsevich and Fukaya, of the $\Ainf$-category of $\Ainf$-functors between to two given $\Ainf$-categories. We briefly recall this construction and state the theorem that the derived enrichment of dg-categories is computed by the (in this case) dg-category of $\Ainf$-functors. We refer to \cite{LH}, \cite{Ly}, \cite{Sei} for a more detailed exposition. 
\medskip
\begin{proposition}\label{p4}
Given unital $\Ainf$-categories $A$ and $B$, there exists an $\Ainf$-category $\Ainf(A,B)$ whose objects are unital $\Ainf$-functors from $A$ to $B$ and graded complex of morphisms between two given $\Ainf$-functors $f$ and $g$ given, in degree $d$, by a sequence of graded morphisms
\medskip
\[
r^d_n: Hom^{\bullet}_{A}(x_{n-1}, x_n)\otimes \dots \otimes Hom^{\bullet}_{A}(x_0, x_1)\to Hom^{\bullet}_{B}(f(x_0),g(x_n))
\]
of degree $d-n$, for $n\ge 0$. 
\end{proposition}
\medskip
\begin{remark}
One can show that if $A$ and $B$ are dg-categories, then $\Ainf(A,B)$ is itself a dg-category. The $m_1$ term of $\Ainf(A,B)$ is given by the graded map
\medskip
\[
m_1: Hom^{\bullet}_{\Ainf(A,B)}(f,g)\to Hom^{\bullet}_{\Ainf(A,B)}(f,g) 
\]
which takes an element $r^d=\lbrace r^d_n \rbrace$ of degree $d$ to the element $m_1(r^d)$ of degree $d+1$ whose $n$-th component is given by the expression
\medskip
\[
m_1(r^d)_n=\sum (-1)^{\ast} m_{r+t+1}(f_{\underline{i}}\otimes r_s \otimes g_{\underline{j}}) - (-1)^d \sum_{u+t+s=n} (-1)^{\ast} r^d_{u+t+1}(Id^{\otimes u}\otimes m_s \otimes Id^{\otimes t})
\]
where $s\ge 0$, $q\ge 0$, $p\ge 0$, $p+q+s=n$, $\underline{i}=(i_1, \cdots, i_r)$, $i_1+\cdots +i_r=p$, $\underline{j}=(j_1, \cdots, j_t)$, $j_1+\cdots +j_t=q$. Morphisms in the category $Z^0(\Ainf(A,B))$ are then identified with natural $\Ainf$-transformations of $\Ainf$-functors \cite{Sei}. The $m_2$ term
\medskip
\[
m_2: Hom^{\bullet}_{\Ainf(A,B)}(f,g)\otimes Hom^{\bullet}_{\Ainf(A,B)}(g,h)\to Hom^{\bullet}_{\Ainf(A,B)}(g,h) 
\]
takes elements $r^{d_1}=\lbrace r^{d_1}_n \rbrace$ and $r^{d_2}=\lbrace r^{d_2}_n \rbrace$, to the element $m_2(r^{d_1},r^{d_2})$ whose $n$-th components are given by the expression
\medskip
\[
m_2(r^{d_1},r^{d_2})_n=\sum (-1)^{\ast} m_s(f_{\underline{i}}\otimes r^{d_1}_s \otimes g_{\underline{j}} \otimes r^{d_2}_u \otimes h_{\underline{l}})
\]
where $\underline{i}=(i_1, \cdots, i_r)$, $i_1+\cdots +i_r=p$, $\underline{j}=(j_1, \cdots, j_t)$, $j_1+\cdots +j_t=q$, $\underline{l}=(l_1, \cdots, l_v)$, $l_1+\cdots +l_v=z$, $p,q,z\ge 0$, $p+s+q+u+z=n$. Here we are not specifying the signs $(-1)^{\ast}$ and we refer to \cite{Sei} for the sign convention adopted.
\end{remark}
\medskip
We restate our first result.
\medskip
\begin{thm}\label{th2}
Given dg-categories $D,E$ there exists natural isomorphisms in $Ho(dgCat)$ 
\[
\mathbb{R}Hom(D,E)\xrightarrow{\simeq}  \mathcal{A}_{\infty}(D,E)
\]
\end{thm}

\subsection{Reminder on $\Ainf$-bimodules.}
In order to give a proof of theorem \ref{th2}, we need to recall the language of $\Ainf$-modules and bimodules. Given a dg-category $E$ there is a dg-category $\Cinf(E)$ whose objects are $\Ainf$-modules on $E$ (see appendix \ref{APPB}). The category $Z^0(\Cinf(E))=Mod_{\infty}(E)$ comes equipped with a notion of weak-equivalences \cite{LH}. As in the case of dg-modules, there is an $\Ainf$-Yoneda embedding, 
\medskip
\[
h^{\infty}:E\to \mathcal{C}_{\infty}(E)
\]
which, for dg-categories, is a dg-functor sending every object of $E$ into its representable dg-module. This dg-functor induces a dg-functor
\medskip
\[
h^{\infty}_{\ast}: \Ainf(D,E)\to \Ainf(D,\Cinf(E)) 
\]
Similarly, if $D$ and $E$ are dg-categories, there exists a dg-category of $\Ainf$-bimodules, $\Cinf(D,E)$ (see appendix \ref{APPB}). We give the following definition.
\medskip
\begin{defi}\label{def1}
An $\Ainf$-bimodule $M$ is called right quasi-representable if, for every $x\in Ob(D)$, the induced $\Ainf$-module $M(x,-)$ is weakly-equivalent in $Mod_{\infty}(E)$ to $h^{\infty}(y(x))$ for some $y(x)\in Ob(E)$.  
\end{defi}
\medskip
A relevant feature of $\Cinf(D,E)$ is that there exists a natural dg-functor
\medskip
\[
z: \Ainf(D,E)\to \Cinf(D,E)
\]
which induces quasi-isomorphisms on each chain complex of morphisms \cite{LH}. Such dg-functor is obtained by composing $h^{\infty}_{\ast}$ with an isomorphism of dg-categories \cite{LH}
\medskip
\[
\Ainf(D,\Cinf(E))\to \Cinf(D,E)
\] 
This construction can be equivalently defined using the notion of $\Ainf$-bifunctors as done in \cite{LyMa}. A result of Lyubashenko and Manzyuk \cite{LyMa} allows to characterize the essential image of $z$.
\medskip
\begin{proposition}
An $\Ainf$-bimodule $M\in  \Cinf(C,D)$ lies in the essential image of $z$ if and only if it is right quasi-representable.
\end{proposition}
\medskip
In paricular, let $\Cinf(D,E)^{rqr}$ be the full dg-subcategory of $\Cinf(D,E)$ whose objects are right quasi-representable $\Ainf$-bimodules, we then have
\medskip
\begin{proposition}\label{pr1}
Given dg-categories $D$ and $E$ the dg-functor $z$ induces natural dg-equivalences
\[
z: \Ainf(D,E)\xrightarrow{\sim} \Cinf(D,E)^{rqr}
\]
\end{proposition}

\subsection{The enveloping dg-category.}
In this section we will describe a dg-functor
\medskip
\[
Mod^{\bullet}(U(D),E)\xrightarrow{} C_{\infty}(D,E)
\]
where $U(D)$ is a particular cofibrant replacement of the dg-category $D$. The restriction of this dg-functor to right quasi-representable bimodules will provide an equivalence of dg-categories
\medskip
\[
Int((Mod^{\bullet}(U(D),E)))^{rqr})\to (\Cinf(D,E))^{rqr}
\]
The combination of this result and proposition \ref{pr1} gives a proof of theorem \ref{th2}. The particular cofibrant replacement $U(D)$ is the enveloping dg-category, that can be defined more generally for any $\Ainf$-category. It has the property that (see appendix \ref{APPA})
\medskip
\[
Hom_{dgCat}(U(A),C)\simeq Hom_{\Ainf Cat}(A,D)
\]
for any $\Ainf$-category $A$ and dg-category $D$. This construction is still meaningful just for dg-categories because it allows to compare dg-bimodules with $\Ainf$-bimodules. 
\medskip
\begin{defi}
Given a dg-category $D$, its enveloping dg-category $U(D)$ is
\medskip
\[
U(D)=(\Omega(B(\overline{D})))^+
\]
where $\Omega$ is the cobar construction, $B$ is the bar construction and $\overline{D}$ is the reduction of the dg-category $D$.
\end{defi}
\medskip
\begin{remark}
$U(D)$ is a dg-category with the same objects of $D$. Each  complex of morphisms in $U(D)$ is the free tensor algebra over the graded vector spaces $B(\overline{D})(x,y)[-1]$ and hence $U(D)$ is a cofibrant dg-category in $(dgCat, Tab)$. Moreover, there is a canonical weak-equivalence  of dg-categories
\medskip
\[
\gamma_D: U(D)\to D
\]
and $(\gamma_D,U(D))$ provides a cofibrant replacement of any dg-category $D$. This is a model for the so called standard resolution $stand(D)$ of \cite{Dr}. The dg-equivalence $\gamma_D$ is determined by its restriction to $B(\overline{D})$ which is itself determined by $Id_{B(\overline{D})}$ projected onto the quiver determined by $\overline{D}$. Explicitly, it sends an object into itself and is defined on morphisms
\medskip
\[
\gamma_D: Hom^{\bullet}_{U(D)}(y_0,y_1)\to Hom^{\bullet}_{D}(y_0,y_1) 
\]
on an element (with abuse of notation) $v_1\otimes \cdots \otimes v_k\in (B(\overline{D})[-1])^{\otimes k}(y_0,y_1)$ by
\medskip
\[
\gamma_D(v_1\otimes \cdots \otimes v_k)=v_1\circ \cdots \circ v_k
\]
if each $v_i\in Hom_{\overline{D}}(y_{i-1},y_i)$ for some pair of objects $(y_{i-1},y_i)$, and $\gamma_D(v_1\otimes \cdots \otimes v_k)=0$ otherwise. It sends moreover the unit into the unit. We discuss now how the enveloping dg-category relates categories of $\Ainf$-bimodules to categories of dg-bimodules. Given $D$ a dg-category, there is a natural commutative diagram of functors
\medskip
\[
 \begin{tikzpicture}
    \def\x{1.5}
    \def\y{-1.2}
    \node (A2_2) at (4*\x, 2.5*\y) {$Mod_{\infty}(D)$};
    \node (A2_1) at (2*\x, 1*\y) {$Mod(U(D))$};
    \node (A1_2) at (0*\x, 2.5*\y) {$CoMod(B^+(D))$};
   
     \path (A2_1) edge [->] node [auto] {$\scriptstyle{J_{D}}$} (A2_2);
     \path (A2_1) edge [->] node [auto,swap] {$\scriptstyle{R_{\tau}(D)}$} (A1_2);
   \path (A2_2) edge [->] node [auto,swap] {$\scriptstyle{B_{D}}$} (A1_2);
          \end{tikzpicture}
\]
Here $Mod(U(D))$ is the category of dg-modules on $U(D)$, $CoMod(B^+(D))$ is the category of dg-comodules over the coaugmented dg-cocategory $B^+(D)$ and $Mod_{\infty}(D)=Z^0(\Cinf(D))$ is the category of $\Ainf$-modules on $D$. Each of those categories comes equipped with a notion of weak-equivalences and those functors induce equivalences on the respective localizations. For more details we refer to appendix \ref{APPB} and to \cite{LH}. We have the following lemma describing the behavior of $J_{D}$ with respect to representable objects.
\end{remark}
\medskip
\begin{lemma}\label{le1}
Consider the composition
\medskip
\[
J_{D}\circ \gamma_D^{\ast}: Mod(D)\to Mod(U(D))\to Mod_{\infty}(D)
\]
where 
\medskip
\[
\gamma_D^{\ast}: Mod(D)\to Mod(U(D))
\]
is the pullback functor along the equivalence $\gamma_D$. Then, the image of a dg-module $M$ is the underlying quiver of $M$ with $\Ainf$-module structure
\medskip
\[
m_i:M(y_0)\otimes Hom_{D}(y_0,y_1)\otimes \cdots \otimes Hom_{D}(y_{i-2},y_{i-1})\to M(y_{i-1}) 
\]
given by
\medskip
\[ 
\left \{
  \begin{tabular}{ccc}
  $m_1(m)=m_1^{M}(m)$ \\
  $m_2(m,\beta_{01})=\sigma(m,\beta_{01})$\\
  $m_i=0,  i>2$
  \end{tabular}
  \right.
\]
where $\sigma$ is the dg-action of $D$ on $M$ and $m_1^{M}$ is the differential of $M$. 
\end{lemma}
\begin{proof}
The $\Ainf$-module $J_{D}\circ \gamma_D^{\ast}(M)$ is determined \cite{LH} by the dg-module $\gamma_D^{\ast}(M)$ by taking its restriction to $B^+(D)[-1]$
\medskip
\[
M\otimes B^+(D)[-1] \xrightarrow{Id\otimes i} M\otimes U(D) \to M 
\]
Such restriction, by the definition of $\gamma_D$, is given by $\sigma(m,v_1)$ for $v_1\in D$ and is $0$ otherwise. Hence, the $\Ainf$-module $J_{D}\circ \gamma_D^{\ast}(M)$ is determined by the composition
\medskip
\[
M\otimes B^+(D)\xrightarrow{(Id\otimes i)\circ (Id\otimes s^{-1})} M\otimes U(D) \to M 
\]
where $s:M\to M[1]$ is the degree $-1$ map of suspension. Such map determines morphisms of degree $1$
\medskip
\[
b_i:M(y_0)\otimes Hom_{D}(y_0,y_1)[1]\otimes \cdots \otimes Hom_{D}(y_{i-2},y_{i-1})[1]\to M(y_{i-1}) 
\]
which are $0$ for $i\ge 2$, and $b_1=m_1^{M}$, $b_2=\sigma(m,s^{-1}(\beta_{01}))$. Suspending and desuspending a suitable number of times we get the result.

\end{proof}

\subsection{Enveloping dg-categories and bimodules.}
Similar constructions exist in the setting of bimodules. Recall \cite{LH} that, for dg-categories $D$ and $E$, there exists a natural commutative diagram 
\medskip
\begin{equation}\label{d1}
 \begin{tikzpicture}
    \def\x{1.5}
    \def\y{-1.2}
    \node (A2_2) at (4*\x, 2.5*\y) {$Mod_{\infty}(D,E)$};
    \node (A2_1) at (2*\x, 1*\y) {$Mod(U(D),U(E))$};
    \node (A1_2) at (0*\x, 2.5*\y) {$CoMod(B^+(D), B^+(E))$};
   
     \path (A2_1) edge [->] node [auto] {$\scriptstyle{J(D,E)}$} (A2_2);
     \path (A2_1) edge [->] node [auto,swap] {$\scriptstyle{R_{\tau}(D,E)}$} (A1_2);
   \path (A2_2) edge [->] node [auto,swap] {$\scriptstyle{B(D,E)}$} (A1_2);
          \end{tikzpicture}
\end{equation}

Here 
\begin{itemize}
\item $Mod(U(D),U(E)))$ is the category of dg-bimodules on $U(D)$ and $U(E)$ \\
\item $CoMod(B^+(D), B^+(E))$ is the category of dg-bicomodules over the coaugmented dg-cocategories $B^+(D)$ and $B^+(E)$ \\
\item $Mod_{\infty}(D,E)=Z^0(\Cinf(D,E))$ is the category of $\Ainf$-bimodules on $D$ and $E$
\end{itemize}
Also in this case, those categories come equipped with a notion of weak-equivalences  with respect to which the functors in the diagram induce equivalences in the localizations. Lemma \ref{le1} has the following corollary.
\medskip
\begin{corollary}\label{cor1}
For any dg-categories $D$ and $E$ the composition 
\medskip
\[
\phi(D,E): Mod(U(D),E)\xrightarrow{(Id_{U(D)}\otimes \gamma_E)^{\ast}} Mod(U(D),U(E)) \xrightarrow{J(D,E)} Mod_{\infty}(D,E)
\]
restricts to a functor
\medskip
\[
Mod(U(D),E)^{rqr}\to Mod_{\infty}(D,E)^{rqr}
\]
\end{corollary} 
\begin{proof}
A bimodule $M\in Mod(U(D),E)$ is right quasi-representable if and only if, for every $x\in Ob(U(D))$, $M(x,-)\in Mod(E)$ is weakly-equivalent to the representable functor $h(y(x))$ for some $y(x)\in E$. This implies that their images in $Mod_{\infty}(E)$ are equivalent, because $J_{E}$ preserves weak-equivalences . By the lemma \ref{le1}, the image in $Mod_{\infty}(E)$ of $h(y(x))$ is itself with higher degree components of the $\Ainf$-module structure equal to $0$, which implies the result. 
\end{proof}
\medskip

We now upgrade the diagram (\ref{d1}) to the level of dg-categories. Recall that the categories of dg-bimodules, $\Ainf$-bimodules and dg-comodules admit a dg-enhancement (see appendix \ref{APPB}). We want to extend the diagram (\ref{d1}) to a commutative diagram of natural dg-functors
\medskip
\[
 \begin{tikzpicture}
    \def\x{1.5}
    \def\y{-1.2}
    \node (A2_2) at (4*\x, 2.5*\y) {$C_{\infty}(D,E)$};
    \node (A2_1) at (2*\x, 1*\y) {$Mod^{\bullet}(U(D),U(E))$};
    \node (A1_2) at (0*\x, 2.5*\y) {$CoMod^{\bullet}(B^+(D), B^+(E))$};
   
     \path (A2_1) edge [->] node [auto] {$\scriptstyle{J(D,E)}$} (A2_2);
     \path (A2_1) edge [->] node [auto,swap] {$\scriptstyle{R_{\tau}(D,E)}$} (A1_2);
   \path (A2_2) edge [->] node [auto,swap] {$\scriptstyle{B(D,E)}$} (A1_2);
          \end{tikzpicture}
\]
where the dg-categories in the diagram are the dg-enhancement of the respective categories of bimodules. The existence of the dg-functor $J(D,E)$ will allow us to define a dg-functor
\medskip
\[
Mod^{\bullet}(U(D),E)\xrightarrow{(Id_{U(D)}\otimes \gamma_E)^{\ast}} Mod^{\bullet}(U(D),U(E)) \xrightarrow{J(D,E)} C_{\infty}(D,E)
\]
which, by corollary \ref{cor1}, will restrict to a dg-functor on the dg-categories of right quasi-representable bimodules
\medskip
\[
(Mod(U(D),E)^{\bullet})^{rqr}\to C_{\infty}(D,E)^{rqr}
\]
Now, the functor $B(D,E)$ comes already from an isomorphism of dg-categories (see appendix \ref{APPB})
\medskip
\[
B(D,E): C_{\infty}(D,E)\to CoMod^{\bullet}(B^+(D), B^+(E))
\]
The functor
\medskip
\[
R_{\tau}(D,E):Mod(U(D),U(E))\to CoMod(B^+(D),B^+(E))
\]
is defined via the indentification 
\medskip
\begin{equation}\label{d2}
 \begin{tikzpicture}
    \def\x{1.5}
    \def\y{-1.2}
    \node (A2_2) at (4*\x, 2.5*\y) {$CoMod(B^+(D)\otimes B^+(E)^{op})$};
    \node (A2_1) at (4*\x, 1*\y) {$CoMod(B^+(D),B^+(E))$};
    \node (A1_2) at (0*\x, 2.5*\y) {$Mod(U(D)\otimes U(E)^{op})$};
    \node(A1_1) at (0*\x, 1*\y) {$Mod(U(D),U(E))$};
   
   \path (A1_1) edge [->] node [auto,swap] {$\scriptstyle{R_{\tau}(D,E)}$} (A2_1);
   \path (A1_1) edge [->] node [auto,swap] {$\scriptstyle{\simeq}$} (A1_2);
   \path (A1_2) edge [->] node [auto,swap] {$\scriptstyle{R_{\tau}}$} (A2_2);
    \path (A2_1) edge [->] node [auto,swap] {$\scriptstyle{\simeq}$} (A2_2);
          \end{tikzpicture}
\end{equation}
where
\medskip
\[
R_{\tau}:Mod(U(D)\otimes U(E)^{op})\to CoMod(B^+(D)\otimes B^+(E)^{op})
\]
sends a dg-module $M\in Mod(U(D)\otimes U(E)^{op})$ in its twisted by $\tau$ tensor product, given by the comodule $M\otimes B^+(D)\otimes B^+(E)^{op}$ with differential twisted by $\tau$ according to the formula
\medskip
\[
b_{\tau}=b_{M}\otimes Id + Id_M \otimes b+(\sigma_M\otimes Id)\circ (Id_M\otimes \tau \otimes Id)\circ (Id_M\otimes \Delta)
\]
Here $\tau$ is a certain acyclic twisted cochain \cite{LH}, $b$, $Id$ and $\Delta$ are differential, identity and cocomposition on $B^+(D)\otimes B^+(E)^{op}$ and $\sigma$ is the dg-action of $U(D)\otimes U(E)^{op}$ on $M$. We have the following proposition.
\medskip
\begin{proposition}
The functor 
\medskip
 \[
 R_{\tau}(D,E): Mod(U(D),U(E))\to CoMod(B^+(D), B^+(E))
 \]
 admits an extension to a dg-functor
 \medskip
 \[
 R_{\tau}(D,E): Mod^{\bullet}(U(D), U(E))\to CoMod^{\bullet}(B^+(D), B^+(E))
 \]
\end{proposition}
\begin{proof}
We show that the functor
\medskip
\[
 R_{\tau}:Mod(U(D)\otimes U(E)^{op})\to CoMod(B^+(D)\otimes B^+(E)^{op})
 \]
admits an extension to a dg-functor. Indeed, this is enough because the vertical arrows of the diagram (\ref{d2}) come from isomorphisms of dg-categories. Fix dg-modules $M_0,M_1$ and consider
\medskip
\[
R_{\tau}: Hom^{\bullet}_{Mod^{\bullet}(U(D)\otimes U(E)^{op})}(M_0,M_1)\to Hom^{\bullet}_{CoMod^{\bullet}(B^+(D)\otimes B^+(E)^{op})}(R_{\tau}(M_0)R_{\tau}(M_1))
\]
given on an element $r$ by
\medskip
\[
R_{\tau}(r)=r\otimes Id
\]
where $Id$ is the identity map of $B^+(D)\otimes B^+(E)^{op}$. It is easy to check that this defines a graded map of degree $0$. This map is compatible with the differential in the sense that
\medskip
\[
d(r)\otimes Id=d(r\otimes Id)
\]
We set
\medskip
\[
t_{\tau}=(\sigma_M\otimes Id)\circ (Id_M\otimes \tau \otimes Id)\circ (Id_M\otimes \Delta)
\] 
and, by definition of the differential, we have
\medskip
\[
\begin{array}{lcl}
d(r\otimes Id)=d_{R_{\tau}(M_1)}\circ (r\otimes Id) - (-1)^{d} (r\otimes Id)\circ d_{R_{\tau}(M_0)}=(d_{M_1}\otimes Id+ \\
\\
+Id_{M_1}\otimes b+t_{\tau})\circ (r\otimes Id) - (-1)^{d}(r\otimes Id)\circ (d_{M_0}\otimes Id+Id_{M_0}\otimes b+t_{\tau})=\\
\\
=(d_{M_1}\circ r)\otimes Id + (-1)^{d} (r\otimes b)+ t_{\tau}\circ (r\otimes Id) - (-1)^{d} (r\circ d_{M_0})\otimes Id -\\
\\
-(-1)^{d} (r\otimes b) -(-1)^{d} (r\otimes Id)\circ t_{\tau}=(d_{M_1}\circ r)\otimes Id + t_{\tau}\circ (r\otimes Id) - \\
\\
-(-1)^{d} (r\circ d_{M_0})\otimes Id -(-1)^{d} (r\otimes Id)\circ t_{\tau}=d(r)\otimes Id + \\
\\
+(-1)^d (\sigma_{M_1}\circ r)\otimes (\tau \otimes Id)\circ \Delta -(-1)^d (r\circ \sigma_{M_0})\otimes (\tau \otimes Id)\circ \Delta=d(r)\otimes Id 
\\
\end{array}
\]
\end{proof}

\subsection{Enveloping dg-category and quasi-representability.}
The existence of an extension of $R_{\tau}(D,E)$ to a dg-functor implies the existence of an extension of the functor $J(D,E)$ to a dg-functor
\medskip
\[
J(D,E): Mod^{\bullet}(U(D),U(E))\to \Cinf(D,E)
\]
The extension of $J(D,E)$ is given by
\medskip
\[
J(D,E)=B(D,E)^{-1}\circ R_{\tau}(D,E)
\]
where $B(D,E)^{-1}$ is the inverse of the dg-isomorphism $B(D,E)$. Consider the composition of dg-functors
\medskip
\[
\phi(D,E): Mod^{\bullet}(U(D),E)\xrightarrow{(Id_{U(D)}\otimes \gamma_E)^{\ast}} Mod^{\bullet}(U(D),U(E)) \xrightarrow{J(D,E)} C_{\infty}(D,E)
\]
Corollary \ref{cor1} implies that this dg-functor restricts to a dg-functor 
\medskip
\[
\phi(D,E): (Mod(U(D),E)^{\bullet})^{rqr}\to C_{\infty}(D,E)^{rqr}
\]
We have the following important proposition relating the derived enrichment described by T\"{o}en in \cite{To} and the dg-category $C_{\infty}(D,E)^{rqr}$.

\medskip
\begin{proposition}\label{pr2}
Given dg-categories $D,E$, the restriction of the dg-functor $\phi(D,E)$ to fibrant and cofibrant right quasi-representable dg-bimodules
\medskip
\[
\phi(D,E): Int((Mod^{\bullet}(U(D),E)))^{rqr})\to (\Cinf(D,E))^{rqr}
\]
is a natural equivalence of dg-categories.
\end{proposition}
\medskip
\begin{proof}
First notice that the dg-functor 
\medskip
\[
Mod^{\bullet}(U(D),E)\xrightarrow{(Id_{U(D)}\otimes \gamma_E)^{\ast}} Mod^{\bullet}(U(D),U(E))
\]
restricts to a dg-equivalence 
\medskip
\[
Int((Mod^{\bullet}(U(D),E)))^{rqr})\xrightarrow{(Id_{U(D)}\otimes \gamma_E)^{\ast}} Int((Mod^{\bullet}(U(D),U(E))))^{rqr})
\]
because it is given by the dg-equivalence
\medskip
\[
\mathbb{R}Hom(D,E)\xrightarrow{(Id_{U(D)}\otimes \gamma_E)^{\ast}} \mathbb{R}Hom(D,U(E))
\]
Moreover $J(D,E)$ induces an equivalence of categories on $H^0$ because we have a commutative diagram (see appendix \ref{APPB})
\medskip
\[
 \begin{tikzpicture}
    \def\x{1.5}
    \def\y{-1.2}
    \node (A2_2) at (4*\x, 2.5*\y) {$Ho((Mod_{\infty}(D,E))^{rqr})$};
    \node (A2_1) at (4*\x, 1*\y) {$H^0((\Cinf(D,E))^{rqr})$};
    \node (A1_2) at (0*\x, 2.5*\y) {$Ho((Mod(U(D),U(E)))^{rqr})$};
    \node(A1_1) at (0*\x, 1*\y) {$H^0(Int((Mod^{\bullet}(U(D),U(E))))^{rqr}))$};
   
   \path (A1_1) edge [->] node [auto,swap] {$\scriptstyle{H^0(J(D,E))}$} (A2_1);
   \path (A1_2) edge [->] node [auto,swap] {$\scriptstyle{\simeq}$} (A1_1);
   \path (A1_2) edge [->] node [auto,swap] {$\scriptstyle{Ho(J(D,E))}$} (A2_2);
    \path (A2_2) edge [->] node [auto,swap] {$\scriptstyle{\simeq}$} (A2_1);
          \end{tikzpicture}
\]
and $Ho(J(D,E))$ is an equivalence of categories. To complete the proof, we need to show that, for every dg-bimodules $M_0,M_1$, the morphism of complexes 
\medskip
\[
J(D,E): Hom_{Mod^{\bullet}(U(D),U(E)))}^{\bullet}(M_0,M_1)\to Hom_{\Cinf(D,E)}^{\bullet}(J(M_0),J(M_1))
\]
induces an isomorphism in cohomology. This is true on $H^0$ because $H^0(J(D,E))$ is an equivalence of categories. Moreover, for a dg-bimodule $M$, there is an obvious structure of dg-bimodule on its shift $M[n]$ and a canonical quasi-isomorphims of complexes
\medskip
\[
Hom_{Mod^{\bullet}(U(D),U(E)))}^{\bullet}(M_0,M_1)[n]\xrightarrow{\sim} Hom_{Mod^{\bullet}(U(D),U(E)))}^{\bullet}(M_0,M_1[n])
\]
The same can be done for $\Ainf$-bimodules, giving a canonical quasi-isomorphism
\medskip
\[
Hom_{\Cinf(D,E))}^{\bullet}(N_0,N_1)[n]\xrightarrow{\sim} Hom_{\Cinf(D,E))}^{\bullet}(N_0,N_1[n])
\]
Those quasi isomorphisms are compatible with $J_{(D,E)}$ in the sense that the diagram
\medskip
\[
\begin{tikzpicture}
    \def\x{1.5}
    \def\y{-1.2}
    \node (A2_2) at (4*\x, 2.5*\y) {$Hom_{\Cinf(D,E))}^{\bullet}(N_0,N_1[n])$};
    \node (A2_1) at (4*\x, 1*\y) {$Hom_{Mod^{\bullet}(U(D),U(E)))}^{\bullet}(M_0,M_1[n])$};
    \node (A1_2) at (0*\x, 2.5*\y) {$Hom_{\Cinf(D,E))}^{\bullet}(N_0,N_1)[n]$};
    \node(A1_1) at (0*\x, 1*\y) {$Hom_{Mod^{\bullet}(U(D),U(E)))}^{\bullet}(M_0,M_1)[n]$};
   
   \path (A1_1) edge [->] node [auto,swap] {$\scriptstyle{\sim}$} (A2_1);
   \path (A1_1) edge [->] node [auto,swap] {$\scriptstyle{J_{(D,E)}[n]}$} (A1_2);
   \path (A1_2) edge [->] node [auto,swap] {$\scriptstyle{\sim}$} (A2_2);
    \path (A2_1) edge [->] node [auto,swap] {$\scriptstyle{J_{(D,E)}}$} (A2_2);
          \end{tikzpicture}
\]
is commutative. Taking $H^0$, we get a commutative diagram 
\medskip
\[
\begin{tikzpicture}
    \def\x{1.5}
    \def\y{-1.2}
    \node (A2_2) at (5*\x, 2.5*\y) {$H^0(Hom_{\Cinf(D,E))}^{\bullet}(N_0,N_1[n]))$};
    \node (A2_1) at (5*\x, 1*\y) {$H^0(Hom_{Mod^{\bullet}(U(D),U(E)))}^{\bullet}(M_0,M_1[n]))$};
    \node (A1_2) at (0*\x, 2.5*\y) {$H^n(Hom_{\Cinf(D,E))}^{\bullet}(N_0,N_1))$};
    \node(A1_1) at (0*\x, 1*\y) {$H^n(Hom_{Mod^{\bullet}(U(D),U(E)))}^{\bullet}(M_0,M_1))$};
   
   \path (A1_1) edge [->] node [auto,swap] {$\scriptstyle{\sim}$} (A2_1);
   \path (A1_1) edge [->] node [auto,swap] {$\scriptstyle{H^n(J_{(D,E)})}$} (A1_2);
   \path (A1_2) edge [->] node [auto,swap] {$\scriptstyle{\sim}$} (A2_2);
    \path (A2_1) edge [->] node [auto,swap] {$\scriptstyle{H^0(J_{(D,E)})}$} (A2_2);
          \end{tikzpicture}
\]
because $H^0(J_{(D,E)})$ is bijective, so it is $H^n(J_{(D,E)})$.

\end{proof}

\subsection{End of the proof of Theorem \ref{th2}}
By the result of T\"{o}en \cite{To} we have natural equivalences of dg-categories 
\medskip
\[
\mathbb{R}Hom(D,E)\xrightarrow{\sim} Int((Mod^{\bullet}(U(D),E)))^{rqr})
\]
Propositions \ref{pr1} and \ref{pr2} provide natural equivalences of dg-categories
\medskip
\[
\begin{gathered}
\phi(D,E): Int((Mod^{\bullet}(U(D),E)))^{rqr})\xrightarrow{\sim} (\Cinf(D,E))^{rqr} \\
z: \Ainf(D,E)\xrightarrow{\sim} \Cinf(D,E)^{rqr}\\
 \end{gathered} 
 \]
The composition of those in $Ho(dgCat)$ gives the required natural isomorphism.

\newpage

\section{The $(\infty,2)$-categories of dg-categories and of $\Ainf$-categories.}
In this section we recall the notions of $(\infty,1)$-category and $(\infty,2)$-category and the $\Ainf$-nerve functor
\medskip
\[
N_{\Ainf}:\Ainf Cat\to SSet
\]
defined in \cite{Fao}, whose values provide examples of $(\infty,1)$-categories. We then define two $(\infty,2)$-categories: the first, $\Ainf Cat_{(\infty,2)}$, has objects the set of $\Ainf$-categories and $(\infty,1)$-category of morphisms given by
\medskip
\[
\Ainf Cat_{(\infty,2)}(A,B)=N_{\Ainf}(\mathcal{A}_{\infty}(A,B))
\]
In this case, we prove the existence of a strictly associative and unital composition law
\medskip
\[
N_{\mathcal{A}_{\infty}}(\Ainf (A,B)) \times N_{\mathcal{A}_{\infty}}(\Ainf (B,C))\to N_{\mathcal{A}_{\infty}}(\Ainf (A,C))
\]
defined using the enrichment in dg-cocategories of the category of $\Ainf$-categories described in \cite{Ly}.
The second $(\infty,2)$-category, $dgCat_{(\infty,2)}$, is obtained from $\Ainf Cat_{(\infty,2)}$ by restricting it to dg-categories. In this case, the $(\infty,1)$-category of morphisms is given by
\medskip
\[
dgCat_{(\infty,2)}(C,D)=N_{\Ainf}(\mathcal{A}_{\infty}(C,D))=N_{dg}(\mathcal{A}_{\infty}(C,D))
\]
where $N_{dg}$ is the dg-nerve of Lurie \cite{LHA}.

\subsection{Simplicial categories and simplicial sets as models for $(\infty,1)$-categories}
We recall two models for $(\infty,1)$-categories and remark the main features and advantages of working with one or the other model. The first model for $(\infty,1)$-categories are weak Kan complexes or quasi-categories.
\medskip
\begin{defi}
A weak-Kan complex is a simplicial set $X$ such that, for any 0$<$ i $<$n and map of simplicial sets $f:\Lambda^{n}_{i} \to X_{\bullet}$, there exists an extension to the full n-simplex  $g:\Delta^{n} \to X_{\bullet}$
\medskip
\[
  \begin{tikzpicture}
    \def\x{1.5}
    \def\y{-1.2}
    \node (A1_1) at (1*\x, 1*\y) {$\Lambda_{i}^{n}$};
    \node (A2_1) at (2*\x, 1*\y) {$X_{\bullet}$};
    \node (A1_2) at (1*\x, 2*\y) {$\Delta^{n}$};
   \path (A1_1) edge [right hook->] node [auto] {$\scriptstyle{i}$} (A1_2);
    \path (A1_1) edge [->] node [auto,swap] {$\scriptstyle{f}$} (A2_1);
    \path (A1_2) edge [->, dashed] node [auto,swap] {$\scriptstyle{g}$} (A2_1);
      \end{tikzpicture}
  \]
where $\Lambda^{n}_{i}$ is the $i$-th inner-horn in $\Delta^{n}$. This property is called left lifting property for inner-horns.
\end{defi}  
Weak-Kan complex are a model for $(\infty,1)$-categories in the following sense:
\medskip
\begin{proposition}
There exists a model category structure $(SSet, Joy)$ on the category of simplicial sets, called the Joyal model structure, for which the fibrant objects are weak Kan complexes. 
\end{proposition} 
\medskip
If $X_{\bullet}$ is an 	$(\infty,1)$-category its $0$-simplicies should be thought as the objects of the $(\infty,1)$-category and the $k$-simplicies as $k$-morphisms. The inner-horn filling property induces a weak composition law, which is associative up to higher degree simplicies and with respect to all $k$-morphisms are invertible for $k>1$ \cite{LHT}. Nevertheless, one would like to work in a model for $(\infty,1)$-categories in which the composition law is strict. This is provided by the second model that we now recall. A simplicial category is a category enriched over the symmetric monoidal category of simplicial sets with monoidal structure given by the cartesian product. $(\infty,1)$-categories can be defined as simplicial categories for which the simplicial set of morphisms between two object is a fibrant Kan complex \cite{LHT}. Such condition encodes the invertibility of $k$-morphisms, for $k\ge1$. More precisely we have:
\medskip
\begin{proposition}
There is a model category structure $(SCat,Berg)$ on the category of simplicial categories, called the Bergner model structure, whose fibrant objects are simplicial categories for which the simplicial set of morphisms between two objects is a Kan complex. 
\end{proposition}
\medskip
Another advantage to work with simplicial categories is that equivalences are easier to describe. Namely, an equivalence of simplicial categories in the Bergner model structure is a functor of simplicial categories $f:\mathcal{C}\to \mathcal{D}$ which induces an equivalence of categories in the associated $0$-homotopy categories and weak-homotopy equivalences on the simplicial sets of morphisms. Those two models for $(\infty,1)$-categories are equivalent in the sense that there exists a pair of adjoint functors  
\[
\adj{ \mathcal{C}[-]}{(SSet,Joy)}{(SCat, Berg)}{N_{SCat}}
\]
which is a Quillen equivalence of model categories \cite{LHT}. The functor $N_{SCat}$ is generally called the homotopy coherent nerve.

\subsection{$(\infty,2)$-categories as preSegal categories in (Sset,Joy)}
In this section we recall a model for $(\infty,2)$-categories due to Lurie \cite{LGo} which is based on the notion of $A$-enriched preSegal category in a model category $(A,M)$. For the purposes of this this paper, we are interested in the case $(A,M)=(SSet,Joy)$ is the model category of simplicial sets with the Joyal model structure. Recall that given a set $S$, $\Delta_{S}$ is the category whose objects are pairs $([n],c)$, where $[n]\in Ob(\Delta)$ and $c:[n]\to S$ is a map of sets, and a morphism $([n],c)\to ([n'],c')$ is a morphism $f:[n]\to [n']$ such that $c=c'\circ f$.
\medskip
\begin{defi}
Let $(A,M)$ be a model category. An $A$-enriched preSegal category is a pair $(S,X)$, where $S$ is a set and $X$ is a functor 
\medskip
\[
X:\Delta^{op}_{S}\to A
\]
such that, for every object $s$, $X[s]$ is the final object in $A$.  
We denote by $Seg_A$ the category whose objects are $A$-enriched preSegal categories and obvious morphism between them.
\end{defi}
\medskip		
A preSegal category gives then a set $S$, that we should think of as the set of objects of $(S,X)$, and, for every collection $s_0,\cdots,s_n$, an element $X[s_0,\cdots ,s_n]$ of $A$ together with maps induced by morphisms in $\Delta_{S}$. The relation with $A$-enriched category theory is better understood via the refined notion of a Segal category.  
\medskip		
\begin{defi}
An $A$-enriched preSegal category $(S,X)$ is a Segal category if, for every sequence of objects $s_0,\cdots,s_n\in S$, the canonical map
\medskip
\[
X[s_0,\cdots,s_n]\to X[s_0,s_1]\times \cdots \times X[s_{n-1},s_n]
\]
exhibits $X[s_0,\cdots,s_n]$ as the homotopy product in $A$ of $\{X[s_{i-1},s_i]\}_{i=1 \cdots n}$.
\end{defi}
\medskip
Every $A$-enriched Segal category $(S,X)$ defines an $Ho(A)$-enriched category $Ho(S,X)$, called the homotopy category of $(S,X)$, whose set of objects is $S$ and morphisms given by
\medskip
\[
Hom_{h(S,X)}(s_0,s_1)=X[s_0,s_1]\in Ho(A)
\]
Composition law is defined by composing the inverse in $Ho(A)$ of the morphism 
\medskip
\[
X[s_0,s_1,s_2]\to X[s_0,s_1]\times X[s_{1},s_2]
\] 
with the canonical map 
\medskip
\[
X[s_0,s_1,s_2]\to X[s_0,s_2]
\]
and unit induced by the degeneracy map 
\medskip
\[
X[s_0]\simeq \ast \to X[s_0,s_0]
\] 
We recall now the notion of a locally fibrant $A$-enriched preSegal category. 
\medskip
\begin{defi}
An $A$-enriched preSegal category $(S,X)$ is locally fibrant if, for every sequence of objects $s_0,\cdots,s_n\in S$, $X[s_0,\cdots,s_n]$ is a fibrant object in $A$.
\end{defi}  
\medskip		
We have the following theorem \cite{LGo}
\medskip		
\begin{thm}
Under suitable hypothesis on the model category $(A,M)$, there exists a model structure on $Seg_A$ of $A$-enriched preSegal categories, called the projective model structure, whose fibrant objects are locally fibrant $A$-enriched Segal categories. 
\end{thm}
\medskip		
We remark that the hypothesis on the model category $(A,M)$ necessary for the theorem to hold are satisfied by the model categories $(SSet,Kan)$ and $(SSet,Joy)$.
\medskip		 
\begin{example}
In the case the model category $(A,M)$ is the category of simplicial sets with the Kan model structure $(SSet,Kan)$, one gets another model for the theory of $(\infty,1)$-categories. More precisely, there exists a Quillen equivalence of model categories
\medskip
\[
\adj{G}{(SCat,Berg)}{(Seg_{(SSet,Kan)}, Proj)}{F}
\]
In particular one can think of $(\infty,1)$-categories as fibrant objects of $(Seg_{(SSet,Kan)}, Proj)$.
\end{example}
\medskip
\begin{example}
In the case the model category $(A,M)$ is the category of simplicial sets with the Joyal model structure $(SSet,Joy)$, we get a model for the theory of $(\infty,2)$-categories. Namely, an $(\infty,2)$-category is a fibrant object of $(Seg_{(SSet,Joy)}, Proj)$, hence, one can think of an $(\infty,2)$-category as a set of objects $S$ and, for any pair of objects $s_0,s_1$, an $(\infty,1)$-category of morphisms $X[s_0,s_1]$, with composition law induced by the correspondence  
\medskip
\[
                 \begin{tikzpicture}
                 \def\x{2.3}
                 \def\y{-1.2}
                 \node (A0_1) at (0*\x, 1*\y) {$X[s_0,s_1]\times X[s_{1},s_2]$};
                 \node (A1_0) at (1*\x, 0*\y) {$X[s_0,s_1,s_2]$};
                 \node (A2_1) at (2*\x, 1*\y) {$X[s_0,s_2]$};
                 \path (A1_0) edge [->] node [auto] {$\scriptstyle{}$} (A0_1);
                 \path (A1_0) edge [->] node [auto,swap] {$\scriptstyle{}$} (A2_1);
                 \end{tikzpicture}
                \]		
In particular, the homotopy category of an $(\infty,2)$-category $Ho(S,X)$ is a (strict) category enriched over $Ho(SSet)_{Joy}$.
\end{example}
\medskip
\begin{remark}[$(\infty,1)$-category associated to an $(\infty,2)$-category]\label{rm5}
Let $(S,X)$ be an $(\infty,2)$-category. Given objects $s_0,s_1$ we have an $(\infty,1)$-category of morphisms $X[s_0,s_1]$ between two objects $s_0,s_1$. We would like to get rid of non-invertible morphisms in this $(\infty,1)$-category in order to get a topological space of morphisms. One way to do this is to consider the largest Kan complex contained in $X[s_0,s_1]$. More precisely, there exists a Quillen pair of adjoint functors \cite{Jo}
\medskip
\[
\adj{i}{(SSet,Kan)}{(SSet, Joy)}{(-)^0}
\]
where $i$ is the inclusion functor. The functor $(-)^0$ associates to a simplicial set $X$, the maximal Kan complex or $\infty$-groupoid, contained in $X$. By standard properties of adjunctions, $(-)^0$ preserves products, fibrant objects and weak-equivalences between fibrant objects and hence this construction allows to define a functor
\medskip
\[
(-)^{\circ}: Fib(Seg_{(SSet,Joy)})\to Seg_{(SSet,Kan)}
\]
by setting
\medskip
\[
(S,X)^{\circ}=(S,X^{\circ})
\]
where, for every sequence of objects $s_0,\cdots,s_n$, 
\medskip
\[
X^{\circ}[s_0,\cdots,s_n]=(X[s_0,\cdots,s_n])^{\circ}
\]
Here $Fib(Seg_{(SSet,Joy)})$ is the full subcategory of fibrant objects of $Seg_{(SSet,Joy)}$.
\end{remark}
\medskip
\begin{defi}
Given an $(\infty,2)$-category $(S,X)\in Fib(Seg_{(SSet,Joy)})$ its associated $(\infty,1)$-category is $(S,X)^{\circ}$.
\end{defi}

\subsection{Nerve construction}
Recall from \cite{Fao}, that there exists a functor, called the nerve construction for $\mathcal{A}_{\infty}$-categories
\medskip
\[
N_{\mathcal{A}_{\infty}}: \mathcal{A}_{\infty} Cat\to SSet
\]
By definition, let $dg[\Delta^n]$ be the dg-category with $Ob(dg[\Delta^n])=\{0,1,\dots ,n\}$ and cochain complex of morphisms
\[ 
Hom^{\bullet}_{dg[\Delta^n]}(i, j) = \left\{
  \begin{array}{l l}
    \mathbb{K}\cdot (i,j) & i\le j\\
    \emptyset & i>j
  \end{array} \right.
\]
where deg($(i, j)$)=$0$ and differential $m_1=0$. Composition is defined as
\[ 
m_2((ij),(jk))=(ik)
\]
for $i\le j\le k$ and $m_2=0$ otherwise.
The construction $[n]\to dg[\Delta^n]$ yields to a functor 
\begin{equation}\label{equ67}
dg[\Delta^{-}]: \Delta \to dgCat
\end{equation}
defining a cosimplicial dg-category. For an $\mathcal{A}_{\infty}$-category $A$, the simplicial set $N_{\mathcal{A}_{\infty}}(A)$ is described by the formula 
\medskip
\[
N_{\mathcal{A}_{\infty}}(A)_n=Hom_{\mathcal{A}_{\infty}Cat}(dg[\Delta^n],A)
\]
with simplicial structure dual to the cosimplicial structure of (\ref{equ67}). One can prove that nerve $N_{\mathcal{A}_{\infty}}(A)$ of any $\mathcal{A}_{\infty}$-category $A$ is an $(\infty,1)$-category \cite{Fao} and that this construction restricts to a functor
\medskip
\[
N_{dg}: dgCat\to SSet
\]
whose values equal the nerve construction for dg-categories of Lurie \cite{LHA}. Recall, that we have an adjunction (see appendix \ref{APPA})
\medskip
\[
\adj{U}{\Ainf Cat}{dgCat}{i}
\]
where $i$ is the canonical inclusion of $dgCat$ in $\Ainf Cat$. Using this adjunction, we find that the $n$-simplicies of the dg-nerve are given by
\medskip
\[
N_{dg}(D)_n=Hom_{\mathcal{A}_{\infty}Cat_{\mathbb{K}}}(dg[\Delta^n],D)=Hom_{dgCat}(U(dg[\Delta^n]),D)
\]
or equivalently
\medskip
\[
Hom_{dgCat}(U(dg[\Delta^n]),D)=Hom_{SSet}(\Delta^n,N_{dg}(D))
\]
this formula suggests that there must exists an extension of the construction $U(dg[-])$, so far defined only for the standard $n$-simplex $\Delta^n$, to a functor
\medskip
\[
U(dg[-]): SSet\to dgCat
\]
which is a left adjoint of the dg-nerve. There is a unique way to do this, namely 
\medskip
\begin{defi}
Define the functor 
\medskip
\[
U(dg[-]): SSet\to dgCat
\]
whose value on a simplicial set $K$ is given by
\medskip
\[
U(dg[K])=colim_{\Delta^n\to K} U(dg[\Delta^n])
\]
\end{defi}
\medskip
\begin{remark}
This functor is well defined because every simplicial set $K$ is isomorphic to the colimit over its $n$-simplicies and it allows to compare $(\infty,1)$-categories with dg-categories in the following sense.
\end{remark}
\medskip
\begin{proposition}\label{pr4}
The functor $U(dg[-])$ defines an adjunction 
\medskip
\[
\adj{U(dg[-])}{SSet}{dgCat}{N_{dg}}
\]
Moreover, this adjunction is a Quillen adjunction of model categories
\medskip
\[
\adj{U(dg[-])}{(SSet,Joy)}{(dgCat,Tab)}{N_{dg}}
\]
\end{proposition}
\begin{proof}
The fact that $U(dg[-])$ defines an adjunction follows from its definition. Moreover, the dg-nerve preserves equivalences and fibrations of dg-categories. To see this fact, one can use the big dg-nerve and the notion of weak-equivalence and fibration in the model category $(SCat,Berg)$.
\end{proof}

\subsection{The $(\infty,2)$-categories $\Ainf Cat_{(\infty,2)}$ and  $dgCat_{(\infty,2)}$}
Recall that \cite{Ly} for $\Ainf$-categories $A,B,C$, there exists a morphism of counital dg-cocategories
\medskip
\[
M: B(\Ainf (A,B)_{nu})^+\otimes B(\Ainf (B,C)_{nu})^+\to  B(\Ainf (A,C)_{nu})^+
\]
which is associative and unital with respect to a morphism of counital dg-cocategries
\medskip
\[
\mathbbm{1}_M: B(\Ainf (A,A)_{nu})^+\to B(\Ainf (A,A)_{nu})^+
\]
Here the $\Ainf$ category $\Ainf (A,B)_{nu}$ is slightly larger than the $\Ainf$-category $\Ainf (A,B)$ in the sense that its objects are not necessarily unital $\Ainf$-functors. However, its restriction to unital $\Ainf$-functors coincides with the $\Ainf$-category $\Ainf (A,B)$ defined in proposition \ref{p4}. We have the following lemma
\medskip
\begin{lemma}
Given $\Ainf$-categories $A,B,C$ the morphisms of dg-cocategories $M$ and $\mathbbm{1}_M$ restrict to morphisms of dg-cocategories
\medskip
\[
\begin{gathered}
M: B(\Ainf (A,B))^+\otimes B(\Ainf (B,C))^+\to  B(\Ainf (A,C))^+ \\
\mathbbm{1}_M: B(\Ainf (A,A))^+\to B(\Ainf (A,A))^+
 \end{gathered} 
 \]
\end{lemma}
\begin{proof}
The proof follows from the fact that $M$ is defined on objects by the composition of $\Ainf$-functors and $\mathbbm{1}_M$ is the identity on objects \cite{Ly}. 
\end{proof}
\medskip
\begin{remark}\label{rm2}
Let $f\in N_{\mathcal{A}_{\infty}}(\Ainf (A,B))_n$, $g\in N_{\mathcal{A}_{\infty}}(\Ainf (B,C))_n$ be $n$-simplicies in the respective nerves. Consider the diagram
\medskip
\begin{equation}\label{dg1}
\begin{tikzpicture}
    \def\x{1.5}
    \def\y{-1.2}
    \node (A2_2) at (5*\x, 2.5*\y) {$B(\Ainf (A,C))^+$};
    \node (A2_1) at (5*\x, 1*\y) {$B(\Ainf (A,B))^+\otimes B(\Ainf (B,C))^+$};
    \node (A1_2) at (0*\x, 2.5*\y) {$B(dg[\Delta^n])^+$};
    \node(A1_1) at (0*\x, 1*\y) {$B(dg[\Delta^n])^+\otimes B(dg[\Delta^n])^+$};
   \path (A1_1) edge [->] node [auto,swap] {$\scriptstyle{B(f)^+\otimes  B(g)^+}$} (A2_1);
   \path (A1_2) edge [->] node [auto,swap] {$\scriptstyle{\Delta_{B(dg[\Delta^n])^+}}$} (A1_1);  
    \path (A2_1) edge [->] node [auto,swap] {$\scriptstyle{M}$} (A2_2);
          \end{tikzpicture}
\end{equation}
where $B(f)^+$ and $B(g)^+$ are the morphisms induced in the bar construction. The composition of those morphisms defines a morphism of counital dg-cocategories 
\medskip
\[
H: B(dg[\Delta^n])^+\to B(\Ainf (A,C))^+
\]
and let $h$ its associated $\Ainf$-functor 
\medskip
\[
h: dg[\Delta^n]\to \Ainf (A,C)
\]
Similarly, for $e\in N_{\mathcal{A}_{\infty}}(\Ainf (A,A))_n$ the composition
\medskip
\[
\begin{tikzpicture}
    \def\x{1.5}
    \def\y{-1.2}
        \node (A2_1) at (3*\x, 1*\y) {$B(\Ainf (A,A))^+$};
    \node (A1_2) at (0*\x, 2.5*\y) {$B(dg[\Delta^n])^+$};
    \node(A1_1) at (0*\x, 1*\y) {$B(\Ainf (A,A))^+$};
   \path (A1_1) edge [->] node [auto,swap] {$\scriptstyle{\mathbbm{1}_M}$} (A2_1);
   \path (A1_2) edge [->] node [auto,swap] {$\scriptstyle{B(e)^+}$} (A1_1);
           \end{tikzpicture}
\]
defines a morphism of counital dg-cocategories
\medskip
\[
K: B(dg[\Delta^n])^+\to B(\Ainf (A,A))^+
\]
and let $k$ its associated $\Ainf$-functor
\medskip
\[
k: dg[\Delta^n]\to \Ainf (A,A)
\]
\end{remark}
\medskip
\begin{lemma}
The $\Ainf$-functors $h$ and $k$ are unital.
\end{lemma}
\begin{proof}
Fix $i\in Ob(dg[\Delta^n])$. Then, under the diagram (\ref{dg1}), the identity at $i$, $Id_i\in B(dg[\Delta^n])^+(i,i)$, is mapped into $Id_{g(i)\circ f(i)}$ and a tensor product of morphism $v_1\otimes \cdots \otimes v_n$, in which at least one of those is $Id_i$, is mapped to $0$ \cite{Ly}. This proves that $h$ is unital. A similar computation shows that $k$ is unital too. 
\end{proof}
This lemma allows to give the following definition.
\medskip
\begin{defi}
Given $\Ainf$-categories $A,B,C$, define the maps of simplicial sets:
\medskip
\[
\mu_{(A,B,C)}: N_{\mathcal{A}_{\infty}}(\Ainf (A,B)) \times N_{\mathcal{A}_{\infty}}(\Ainf (B,C))\to N_{\mathcal{A}_{\infty}}(\Ainf (A,C))
\]
that on $n$-simplicies $f$ and $g$ is given by 
\medskip
\[
\mu_{(A,B,C)}(f,g)=h 
\]
where $h$ is defined in remark \ref{rm2} and
\medskip
\[
\mathbbm{1}_{A}: N_{\mathcal{A}_{\infty}}(\Ainf (A,A))\to N_{\mathcal{A}_{\infty}}(\Ainf (A,A))
\]
that, on an $n$-simplex $e$, is given by
\medskip
\[
\mathbbm{1}_{A}(e)=k
\]
where $k$ is defined in remark \ref{rm2}.
\end{defi}
\medskip
\begin{proposition}
The maps of simplicial sets 
\medskip
\[
\begin{gathered}
\mu_{(A,B,C)}: N_{\mathcal{A}_{\infty}}(\Ainf (A,B)) \times N_{\mathcal{A}_{\infty}}(\Ainf (B,C))\to N_{\mathcal{A}_{\infty}}(\Ainf (A,C)) \\
\mathbbm{1}_A: N_{\mathcal{A}_{\infty}}(\Ainf (A,A))\to N_{\mathcal{A}_{\infty}}(\Ainf (A,A))
 \end{gathered} 
 \]
satisfy the identities:
\[ 
\left \{
  \begin{tabular}{ccc}
  $\mu_{(A,C,D)}\circ (\mu_{(A,B,C)}\otimes Id_{\Ainf (C,D)})= \mu_{(A,B,D)}\circ (Id_{\Ainf (A,B)}\otimes \mu_{(B,C,D)})$ \\
  $\mu_{(A,A,B)}\circ (\mathbbm{1}_A\otimes Id_{\Ainf (A,B)})=Id_{\Ainf (A,B)}$ \\
  $\mu_{(A,B,B)}\circ (Id_{\Ainf (A,B)}\otimes \mathbbm{1}_B)=Id_{\Ainf (A,B)}$
  \end{tabular}
  \right.
\]
\medskip
\end{proposition}
\begin{proof}
It follows from the coassociativity of $\Delta_{B^+(dg[\Delta^n])}$, the associativity of tensor product $\otimes$ and associativity and unitality of $M$ and $\mathbbm{1}_M$.

\end{proof}
\medskip
We have at this point all the necessary to define the $(\infty,2)$-category of $\Ainf$-categories.
\medskip
\begin{defi}\label{def8}
The $(\infty,2)$-category of $\Ainf$-categories, denoted by $\Ainf Cat_{(\infty,2)}$, is the $(SSet,Joy)$-preSegal category $(S,X)$ defined as:
\medskip
\[
S=Ob(\Ainf Cat)
\] 
and
\medskip
\[
X:\Delta^{op}_{S}\to SSet
\]
given by 
\medskip
\[
X(a_0)=\ast
\] 
and, for a sequence of $\Ainf$-categories $(a_0,\cdots, a_n)$, by
\medskip
\[
X(a_0,\cdots, a_n)=N_{\mathcal{A}_{\infty}}(\Ainf (a_0,a_1))\times \cdots N_{\mathcal{A}_{\infty}}(\Ainf (a_{n-1},a_n))
\]
The induced maps of simplicial sets, generated by the morphisms in $\Delta_{S}^{op}$
\medskip
\[ 
\left \{
  \begin{tabular}{ccc}
  $\sigma^n_j: (a_0,\cdots, a_j, a_j, \cdots, a_n)\to (a_0,\cdots, a_n)$ \\
  \\
  $d^n_j: (a_0,\cdots, \hat{a}_j, \cdots, a_n)\to (a_0,\cdots, a_n)$
  \end{tabular}
  \right.
\]
are given:
\begin{itemize}
 \item 
 for $0\le j\le n$
\medskip
\[
X(\sigma^n_j):X(a_0,\cdots, a_n) \to X(a_0,\cdots, a_j, a_j, \cdots, a_n)
\]
by
\medskip
\[
X(\sigma^n_j)=Id^{j}_{}\times \mathbbm{1}_{a_j} \times Id^{(n-j-1)}
\]
\item
for $1\le j\le n-1$ 
\medskip
\[
X(d^n_j):X(a_0,\cdots, a_n) \to X(a_0,\cdots, \hat{a}_j, \cdots, a_n) 
\]
by
\medskip
\[
X(d^n_j)=Id^{(j-1)}_{}\times \mu_{(a_{j-1},a_j,a_{j+1})} \times Id^{(n-j)}
\]
\item
for $j=0,n$, by
\medskip
\[ 
\left \{
  \begin{tabular}{ccc}
  $X(d^n_0)=\pi_{N_{\mathcal{A}_{\infty}}(\Ainf (a_1,a_2))\times \cdots N_{\mathcal{A}_{\infty}}(\Ainf (a_{n-1},a_n))}$ \\
  \\
  $X(d^n_n)=\pi_{N_{\mathcal{A}_{\infty}}(\Ainf (a_0,a_1))\times \cdots N_{\mathcal{A}_{\infty}}(\Ainf (a_{n-2},a_{n-1}))}$
  \end{tabular}
  \right.
\]
where $\pi$ are the relative projections.
\end{itemize}
\end{defi}
\medskip
\begin{proposition}
The above definition defines a fibrant Segal category in the model category $(SSet, Joy)$, hence an $(\infty,2)$-category.
\end{proposition}
\begin{proof}
The fact that $X$ defines a functor follows from the associativity and unitality of $\mu$. Clearly, each of the simplicial sets in the image of $X$ is fibrant in $(SSet, Joy)$ because they are $(\infty,1)$-categories. Moreover the map of simplicial sets induced by $X$
\medskip
\[
X(a_0,\cdots,a_n)\to X(a_0,a_1)\times \cdots \times X(a_{n-1},a_n)
\]
is the identity and hence, being all the objects fibrant, it exhibits $X(a_0,\cdots,a_n)$ as the homotopy product in $(SSet,Joy)$ of $\{X(a_{i-1},a_i)\}_{i=1 \cdots n}$.

\end{proof}
\medskip
\begin{defi}
The $(\infty,2)$-category of dg-categories, denoted by $dgCat_{(\infty,2)}$, is the $(SSet,Joy)$-preSegal category $(T,Y)$ defined as:
\medskip
\[
T=Ob(dgCat)
\] 
and
\medskip
\[
Y:\Delta^{op}_{T}\to SSet
\]
given by the composition
\medskip
\[
\Delta^{op}_{Ob(dgCat)}\xrightarrow{} \Delta^{op}_{Ob(\Ainf Cat)}\xrightarrow{X} SSet
\]
where the first arrow is induced by the obvious inclusion $i: Ob(dgCat)\to Ob(\Ainf Cat)$ and $X$ is the functor defined in \ref{def8}
\end{defi}	
\medskip
It is clear that the above definition defines a fibrant Segal category in the model category $(SSet, Joy)$, hence an $(\infty,2)$-category.
\begin{remark}
For dg-categories $d_0,d_1$, the simplicial set $Y(d_0,d_1)$ equals, by definition, the dg-nerve of the dg-category of (unital) $\Ainf$-functors
\medskip
\[
Y(d_0,d_1)=N_{dg}(\Ainf (d_0,d_1))
\]
The Segal category structure defines a strictly associative composition maps on the dg-nerves
\medskip
\[
\mu_{(d_0,d_1,d_2)}: N_{dg}(\Ainf (d_0,d_1)) \times N_{dg}(\Ainf (d_1,d_2))\to N_{dg}(\Ainf (d_0,d_2))
\]
and units
\medskip
\[
\mathbbm{1}_{d_0}: N_{dg}(\Ainf (d_0,d_0))\to N_{dg}(\Ainf (d_0,d_0))
\] 
\end{remark}
\newpage
\section{The $(\infty,1)$-category of dg-categories as a model for simplicial localization. }
In this section we introduce the $(\infty,1)$-category of dg-categories $dgCat_{(\infty,1)}$, which is defined as the $(\infty,1)$-category associated to $dgCat_{(\infty,2)}$. We prove that its mapping spaces of morphisms are weakly-homotopy equivalent to the mapping spaces in the Dwyer-Kan localization \cite{DK} $L_{Tab}(dgCat)$ and hence $dgCat_{(\infty,1)}$ should be understood as the correct $(\infty,1)$-category associated to the model category $(dgCat,Tab)$.
\subsection{$dgCat_{(\infty,1)}$ as a model for the simplicial localization}
\medskip
\begin{defi}
The $(\infty,1)$-category of dg-categories, denoted by $dgCat_{(\infty,1)}$, is given by
\medskip
\[
dgCat_{(\infty,1)}=(dgCat_{(\infty,2)})^{\circ}
\]
where $(dgCat_{(\infty,2)})^{\circ}$ is the associated $(\infty,1)$-category to an $(\infty,2)$-category as by remark \ref{rm5}.
\end{defi}
By the definition of $dgCat_{(\infty,2)}$, $dgCat_{(\infty,1)}$ is a genuine fibrant simplicial category. We have the following theorem.
\medskip
\begin{thm}
Given dg-categories $C, D$ there exists weak homotopy equivalences of simplicial sets
\medskip
\[
Map_{Ho(dgCat)}(C,D) \xrightarrow{} Map_{dgCat_{(\infty,1)}}(C,D) 
\]
where $Map_{Ho(dgCat)}(C,D)$ is the mapping space in $L_{Tab}(dgCat)$.
\end{thm}
\begin{proof}
Recall that from \ref{pr4}, we have a Quillen adjunction of model categories 
\medskip
\[
\adj{U(dg[-])}{(SSet,Joy)}{(dgCat,Tab)}{N_{dg}}
\]
which induces natural weak-equivalences 
\medskip
\[
Map_{Ho(dgCat)}(U(dg[K]),D)\xrightarrow{\sim} Map_{Ho(SSet)_{Joy}}(K,N_{dg}(D))
\]
Consider the dg-category $dg[\Delta^0]$. This is a cofibrant dg-category and it is equal, by construction, to its enveloping dg-category $U(dg[\Delta^0])$. This implies that, for every dg-category $C$, we have an equivalence of dg-categories, 
\medskip
\[
dg[\Delta^0]\otimes ^{\mathbb{L}} C\xrightarrow{\sim}  dg[\Delta^0]\otimes Q(C) \xrightarrow{\sim} C
\]
This dg-equivalence, induces an homotopy equivalence in the mapping spaces
\medskip
\[
Map_{Ho(dgCat)}(C,D)\xrightarrow{\sim} Map_{Ho(dgCat)}(dg[\Delta^0]\otimes ^{\mathbb{L}} C,D)
\]
The closed symmetric monoidal structure on $Ho(dgCat)$ and theorem \ref{th2} gives natural weak-equivalences
\medskip
\[
Map_{Ho(dgCat)}(dg[\Delta^0]\otimes ^{\mathbb{L}} C,D) \xrightarrow{\sim} Map_{Ho(dgCat)}(U(dg[\Delta^0]), \Ainf (C,D))
\]
and, again, the adjunction $(U(dg[-]),N_{dg})$, provides natural weak-equivalences 
\medskip
\[
Map_{Ho(dgCat)}(U(dg[\Delta^0]), \Ainf (C,D)) \xrightarrow{\sim} Map_{Ho(SSet)_{Joy}}(\Delta^0,N_{dg}(\Ainf (C,D)))
\]
The mapping space $Map_{Ho(SSet)_{Joy}}(\Delta^0,N_{dg}(\Ainf (C,D)))$ is homotopy equivalent to the mapping space in the $(\infty,1)$-category of $(\infty,1)$-categories \cite{LHT} and hence we have equivalences
\medskip
\[
Map_{Ho(SSet)_{Joy}}(\Delta^0,N_{dg}(\Ainf (C,D)))\xrightarrow{\sim} (Map_{SSet}(\Delta^0,N_{dg}(\Ainf (C,D))))^{\circ}
\]
but clearly
\medskip
\[
(Map_{SSet}(\Delta^0,N_{dg}(\Ainf (C,D))))^{\circ}=(N_{dg}(\Ainf (C,D)))^{\circ}
\]
which completes the proof.
\end{proof}

\newpage
	
\section{Application: Hochschild Cohomology for $\Ainf$-categories.}
There are two equivalent approaches to define the Hochschild cohomology of a dg-category. The first approach is through derived functors, namely, for a dg-category $C$, consider the model category of dg-modules over $C\otimes^{\mathbb{L}} C^{op}$. The dg-category $C$ is identified with the representable dg-module
\medskip
\[
C(x,y)=Hom_{C}(x,y)
\]
The Hochschild complex is defined as 
\medskip
\[
\mathbb{HH}(C,C)=\mathbb{R}Hom^{\bullet}_{Mod(C\otimes^{\mathbb{L}} C^{op})}(C,C)
\] 
and its cohomology, denoted by $HH^{\bullet}(C,C)$, is the Hochschild cohomology of $C$. Because the model category $Mod(C\otimes^{\mathbb{L}} C^{op})$ is $Ch(\mathbb{K})$-enriched, the result of \cite{To} implies the existence of an equivalence of complexes
\medskip
\[
\mathbb{R}Hom^{\bullet}_{Mod(C\otimes^{\mathbb{L}} C^{op})}(C,C)\simeq Hom_{\mathbb{R}Hom(C,C)}^{\bullet}(C,C)
\]
On the other hand, one can take a resolution of $C$ as $C\otimes^{\mathbb{L}} C^{op}$-module and get an explicit model for the Hochschild complex. As showed in \cite{FMT}, there exists a suitable resolution for which one gets an equivalence of complexes
\medskip
\begin{equation}\label{eq34}
\mathbb{HH}(C,C)\simeq Hom_{\Ainf (C,C)}^{\bullet}(Id_C,Id_C)
\end{equation}
The first result of this paper reconciles the two approaches showing that the dg-categories $\Ainf (C,C)$ and $\mathbb{R}Hom(C,C)$ are equivalent. The construction of the $(\infty,2)$-category $dgCat_{(\infty,2)}$ allows to give a topological interpretation of the Hochschild complex. Namely, consider the $(\infty,1)$-category of endomorphisms of $C$ in $dgCat_{(\infty,2)}$ given by $N_{dg}(\Ainf (C,C))$. We can associate to it a fibrant simplicial category, (via the homotopy coherent nerve adjunction for instance) and extract a topological space of maps at $Id_C$, that we denote by
\medskip
\begin{equation}
End_{dgCat_{(\infty,2)}}(Id_C)=Map_{N_{dg}(\Ainf (C,C))}(Id_C,Id_C)
\end{equation}
The homotopy groups of this space are related to the Hochschild cohomology of the dg-category $C$ in the following sense.
\medskip
\begin{proposition}
For any dg-category $C$, $i\ge 0$ we have
\medskip
\[
\pi_i(Map_{N_{dg}(\Ainf (C,C))}(Id_C,Id_C))=HH^{-i}(C,C)
\]
\end{proposition}
\begin{proof}
The proof follows from the fact that the $i$-th homotopy groups of the mapping space are the $-i$-th cohomology of the complex $Hom^{\bullet}_{\Ainf (C,C)}(Id_C,Id_C)$ which by (\ref{eq34}) are equal to $-i$-th Hochschild cohomology of $C$.

\end{proof}
\medskip
The approach via explicit resolutions extends the definition of Hochschild cohomology to $\Ainf$-categories \cite{Ke2}, \cite{Tr}.
\medskip
\begin{defi}\label{defi87}
Given an $\Ainf$-category $A$, its Hochschild complex is
\medskip
\[
\mathbb{HH}(A,A)=Hom^{\bullet}_{\Ainf (A,A)}(Id_A,Id_A)
\]
and its cohomology $HH^{\bullet}(A,A)$ is the Hochschild cohomology of the $\Ainf$-category $A$.
\end{defi}
\medskip
In analogy with the dg-case, we can consider the $(\infty,1)$-category of endomorphisms of $A$ in $\Ainf Cat_{(\infty,2)}$, namely $N_{\Ainf}(\Ainf (A,A))$, and extract a topological space of maps at $Id_A$
\medskip
\[
End_{\Ainf Cat_{(\infty,2)}}(Id_A)=Map_{N_{\Ainf}(\Ainf (A,A))}(Id_A,Id_A)
\]
This can be done via the homotopy coherent nerve adjunction or, equivalently, taking the left mapping space \cite{LHT}. In this setting, we can generalize the result of \cite{To} of computation of Hochschild cohomology for dg-categories to the context of $\Ainf$-categories where the technique of derived enrichment do not apply for the lack of a model category of $\Ainf$-categories \cite{LH}. We have the folowing theorem.
\medskip
\begin{theorem}
For any $\Ainf$-category $A$, $i\ge 0$, we have
\medskip
\[
\pi_i(End_{\Ainf Cat_{(\infty,2)}}(Id_A))=HH^{-i}(A,A)
\]
\end{theorem}
\begin{proof}
Using the the left mapping space \cite{LHT} one easily sees that the homotopy groups of $End_{\Ainf Cat_{(\infty,2)}}(Id_A)$ are given by
\medskip
\[
\pi_i(End_{\Ainf Cat_{(\infty,2)}}(Id_A))=H^{-i}(Hom^{\bullet}_{\Ainf (A,A)}(Id_A,Id_A))
\]
The right hand side of this equation, by \ref{defi87} equals $HH^{-i}(A,A)$

\end{proof}
\newpage
\subsection*{Final remarks}
The Hochschild complex of an $\Ainf$-category has a structure of $B_{\infty}$-algebra \cite{GeJo}. This means that 
\medskip
\[
B=B(Hom^{\bullet}_{\Ainf (A,A)}(Id_A,Id_A))
\] 
the (unreduced) bar construction of $Hom^{\bullet}_{\Ainf (A,A)}(Id_A,Id_A)$, comes equipped with a differential
\medskip
\[
b:B\to B
\] 
a multiplication
\medskip
\[
\mu:B\otimes B\to B
\]
and a unit
\medskip
\[
\varepsilon: \mathbb{K}\to B
\]
making $(B,b,\Delta, \nu, \mu, \varepsilon)$ a unital-counital-dg-bialgebra. Under the identification
\medskip
\[
\mathbb{HH}(A,A)\simeq Hom^{\bullet}_{\Ainf (A,A)}(Id_A,Id_A)
\]
the $B_{\infty}$-algebra structure induces maps of complexes
\medskip
\[
\begin{gathered}
m_2: \mathbb{HH}(A,A)\otimes \mathbb{HH}(A,A)\to \mathbb{HH}(A,A) \\
\mu_{1,1}: \mathbb{HH}(A,A)\otimes \mathbb{HH}(A,A)\to \mathbb{HH}(A,A)[1]
\end{gathered} 
\]
\medskip
Those maps define a Gerstenhaber algebra structure on the Hochschild cohomology, with Gerstenhaber bracket 
\medskip
\[
[-,-]: HH^{p}(A,A)\otimes HH^{q}(A,A)\to HH^{p+q+1}(A,A)
\]
given by
\medskip
\[
[a,b]= \mu_{1,1}(a,b)-(-1)^{(deg(a)+1)(deg(b)+1)} \mu_{1,1}(b,a)
\]
and cup product induced by $m_2$. The following proposition relates this construction with the enrichment of $\Ainf$-categories over dg-cocategories.
\medskip
\begin{proposition}\cite{Ke2}
Given an $\Ainf$-category $A$, the multiplication $\mu$ and unit $\varepsilon$ on $B$ are given by the restriction of the dg-cocategory morphisms
\medskip
\[
\begin{gathered}
M: B(\Ainf(A,A)\otimes B(\Ainf(A,A))\to B(\Ainf(A,A)) \\
\mathbbm{1}_M: B(\Ainf(A,A)\to B(\Ainf(A,A))
\end{gathered} 
\]
to the complex $B$.
\end{proposition}
\medskip
This proposition suggests that the $(\infty,2)$-categories $dgCat_{(\infty,2)}$ and $\Ainf Cat_{(\infty,2)}$ encode the $B_{\infty}$-algebra structure of the Hochschild complex. Namely, the topological space $End_{\Ainf Cat_{(\infty,2)}}(Id_A)$ of endomorphisms of $Id_A$, comes equipped with two maps
\medskip
\[
\begin{gathered}
m_2: End_{\Ainf Cat_{(\infty,2)}}(Id_A)\times End_{\Ainf Cat_{(\infty,2)}}(Id_A)\to End_{\Ainf Cat_{(\infty,2)}}(Id_A) \\
\mu: End_{\Ainf Cat_{(\infty,2)}}(Id_A)\times End_{\Ainf Cat_{(\infty,2)}}(Id_A)\to End_{\Ainf Cat_{(\infty,2)}}(Id_A)
\end{gathered} 
\]
which appears as a by-product of the $(\infty,2)$-category structure of $\Ainf Cat_{(\infty,2)}$ and are related to the maps $m_2$ and $\mu_{1,1}$. However, the homotopy groups of these topological spaces just partially compute the Hochschild cohomology, which does not seem to be completely satisfactory. The author believes that a suitable notion of stable $(\infty,2)$-category could provide a solution to this issue. Such notion, indeed, will provide a spectra of morphisms $Sp(Id_A,Id_A)$, whose homotopy groups will then compute the full Hochschild cohomology. The maps $m_2$ and $\mu$ should appear as truncations of maps of spectra
\medskip
\[
\begin{gathered}
m_2: Sp(Id_A,Id_A)\wedge Sp(Id_A,Id_A)\to Sp(Id_A,Id_A) \\
\mu: Sp(Id_A,Id_A)\wedge Sp(Id_A,Id_A)\to Sp(Id_A,Id_A)[1]
\end{gathered} 
\]
out of which the Gerstenhaber structure of the Hochschild cohomology appears by taking the associated maps in the homotopy groups. Those observations, restricted to the setting of dg-categories, can possibly provide an answer to the question "What do DG categories form?". This question was posed and discussed in \cite{Tam}.
\newpage

\appendix
\section{dg-categories and $\Ainf$-categories.}\label{APPA}

Let $\mathbb{K}$ be a field, that we assume from now on of characteristic $0$. The category $Vect_{\mathbb{Z}}(\mathbb{K})$ is the category whose objects are $\mathbb{Z}$-graded vector spaces over $\mathbb{K}$
\medskip
\[
V^{\bullet}=\bigoplus_{p\in \mathbb{Z}} V^p
\]
and morphisms are given by degree preserving $\mathbb{K}$-linear maps. This category has a closed symmetric monoidal structure, with monoidal functor given by tensor product of graded vector spaces. We refer to \cite{LH} for details about their definition.

Let $S$ be a set, the category of graded quivers on $S$, denoted by $Qu(S,Vect_{\mathbb{Z}}(\mathbb{K}))$, is the category whose objects are collections of graded vector spaces
\medskip
\[
\mathcal{Q}=\lbrace \mathcal{Q}(x,y) \rbrace_{x,y\in S}
\]
and morphisms $r:\mathcal{Q}\to \mathcal{R}$ are collections of maps of graded vector spaces
\medskip
\[
\lbrace r(x,y): \mathcal{Q}(x,y)\to \mathcal{R}(x,y) \rbrace_{x,y\in S}
\]
Given $\mathcal{Q}$ and $\mathcal{R}\in Qu(S,Vect_{\mathbb{Z}}(\mathbb{K}))$, their tensor product $\mathcal{Q}\otimes \mathcal{R}$ is the graded quiver on $S$
\medskip
\[
\mathcal{Q}\otimes \mathcal{R}(x,y)=\bigoplus_{z\in S}  \mathcal{Q}(x,z)\otimes \mathcal{R}(z,y)
\]
The quiver $\mathbb{K}\in Qu(S,Vect_{\mathbb{Z}}(\mathbb{K}))$ is the quiver given by $\mathbb{K}(x,y)=\mathbb{K}$ for $x=y\in S$ and $0$ otherwise.

A chain complex is a graded vector space together with a map of degree $+1$ which squares to $0$. They form a category, denoted by $Ch(\mathbb{K})$, where morphisms are morphisms of graded vector spaces compatible with differentials. As for $Vect_{\mathbb{Z}}(\mathbb{K})$, there is a symmetric monoidal structure on $Ch(\mathbb{K})$ \cite{LH}.

Let S be a set, the category of dg-quivers on $S$, denoted $Qu(S,Ch(\mathbb{K}))$ is the category whose objects are graded quivers endowed with a differential and morphisms are morphisms of graded quivers compatible with the differentials. 
Given $\mathcal{Q}, \mathcal{R}\in Qu(S,Qu(S,Ch(\mathbb{K})))$, their tensor product is the tensor product as graded quivers with differential $d_{\mathcal{Q}\otimes \mathcal{R}}=d_{\mathcal{Q}}\otimes Id_{\mathcal{R}}+Id_{\mathcal{Q}}\otimes d_{\mathcal{R}}$.

\medskip

\subsection{Differential graded categories.}
A dg-category $D$ over $\mathbb{K}$ is a category enriched over the monoidal category $Ch(\mathbb{K})$. It is given by a set of objects $Ob(D)$ and, for every pair of objects $x,y$, a chain complex $Hom^{\bullet}_{D}(x,y)$ with composition morphisms  
\medskip
\[
Hom^{\bullet}_{D}(y,z)\otimes Hom^{\bullet}_{D}(x,y)\to Hom^{\bullet}_{D}(x,z)
\]
which are associative and unital.

Given dg-categories $C,D$, a dg-functor $f:C\to D$ is a map of sets $f:Ob(C)\to Ob(D)$ and, for every pair of objects $x,y$, a map of chain complexes 
\medskip
\[
f_{x,y}: Hom^{\bullet}_{C}(x,y)\to Hom^{\bullet}_{D}(f(x),f(y))
\] 
compatible with the composition morphisms in the obvious way and preserving the identities. There is an obvious composition law for dg-functors that is associative and unital. We refer to $dgCat$ as the category whose objects are dg-categories and morphisms are dg-functors.

Given dg-categories $C,D$ and dg-functors $f,g:C\to D$, a morphism of dg-functors (or a natural transformation) between $f$ and $g$ is the data, for every $x\in Ob(C)$, of a morphism of graded vector spaces
\medskip
\[
r(x):\mathbb{K}\to Hom^{\bullet}_{D}(f(x),g(x))
\]
such that
\medskip
\[ 
\left \{
  \begin{tabular}{ccc}
  $d_D\circ r(x)=0$ \\
  \\
  $m_{f(x),g(x),g(y)}\circ(r(x)\otimes g_{x,y})=m_{f(x),f(y),g(y)}\circ (f_{x,y}\otimes r(y))$
  \end{tabular}
  \right.
\]
\medskip
\begin{remark}
As well known, dg-categories with a given set of objects $S$ are identified with unital dg-algebra objects in the monoidal category $Qu(S,Vect_{\mathbb{Z}}(\mathbb{K}))$. A non-unital dg-category over a set of objects $S=Ob(D)$ is a non-unital dg-algebra object in the monoidal category $Qu(S,Vect_{\mathbb{Z}}(\mathbb{K}))$. We denote by $dgCat_{nu}$ the category of non-unital dg-categories.
\end{remark}
\medskip
\begin{defi}
Given a dg-category $D\in dgCat$, its homotopy category $H^0(D)$ is the category with the same objects of $D$ and set of morphisms
\medskip
\[
Hom_{H^0(D)}(x,y)=H^0(Hom^{\bullet}_{D}(x,y))
\] 
composition law and identities are induced in cohomology by the one of $D$.
\end{defi}

\subsection{Tensor product and dg-category of dg-functors}
\begin{defi}
Given dg-categories $C,D$, their tensor product $C\otimes D$ is the dg-category whose objects 
\medskip
\[
Ob(C\otimes D)=Ob(C)\times Ob(D)
\]
and cochain complex of morphism given by
\medskip
\[
Hom^{\bullet}_{C\otimes D}((x_1,y_1),(x_2,y_2))=Hom^{\bullet}_{C}(x_1,x_2)\otimes Hom^{\bullet}_{D}(y_1,y_2)
\]
with differential
\medskip
\[
d_{C\otimes D}=d_C\otimes Id_D+Id_C\otimes d_D
\]
Composition law and identity are obviously defined.
\end{defi}
\medskip
\begin{defi}
Given dg-categories $C,D$, the dg-category of dg-functors $dgFun^{\bullet}(C,D)$, is the dg-category whose objects are dg-functors $f:C\to D$  and, give dg-functors $f$ and $g$ an element $r\in Hom^{d}_{dgFun^{\bullet}(C,D)}(f,g)$ is given by a sequences of morphisms of degree $d$
\medskip
\[
r(x):\mathbb{K}\to Hom^{\bullet}_{D}(f(x),g(x))
\]
for every $x\in Ob(C)$, such that
\medskip
\[
m_{f(x),g(x),g(y)}\circ(r(x)\otimes g_{x,y})=m_{f(x),f(y),g(y)}\circ (f_{x,y}\otimes r(y))
\]
The differential 
\medskip
\[
d:Hom^{d}_{dgFun^{\bullet}(C,D)}(f,g)\to Hom^{d+1}_{dgFun^{\bullet}(C,D)}(f,g)
\]
is given by the formula $d(r)(x)=d_D(r(x))$.
\end{defi}
\medskip
\begin{proposition}
The tensor product defines a symmetric monoidal structure on $dgCat$. This monoidal structure in closed. In particular there exists natural isomorphisms
\medskip
\[
Hom_{dgCat}(C\otimes D, E)\xrightarrow{\simeq} Hom_{dgCat}(C,dgFun^{\bullet}(D,E))
\]
for given dg-categories $C,D,E$.
\end{proposition}

\subsection{$\mathcal{A}_{\infty}$-categories and differential graded cocategories}
\begin{defi}
A (unital) $\Ainf$-category $A$ over $\mathbb{K}$ is the data of a set objects $Ob(A)$, a graded quiver $\lbrace Hom_{A}^{\bullet}(x,y)\rbrace_{x,y\in Ob(A)} $, and graded morphisms of degree $2-k$
\medskip
\[
m_k: Hom^{\bullet}_{A}(x_{k-1}, x_k)\otimes \dots \otimes Hom^{\bullet}_{A}(x_0, x_1)\to Hom^{\bullet}_{A}(x_0, x_k)
\]
$k\ge 1$, satisfying the system of equations
\medskip
\begin{equation}
\sum_{n=i+j+k} (-1)^{ik+j} m_{i+j+1}(Id^{\otimes^i}\otimes m_k \otimes Id^{\otimes^j})=0
\end{equation}
for $n\ge 1$. Moreover, it comes equipped with a map of degree $0$ 
\medskip
\[
\varepsilon_{x}: \mathbb{K}\to Hom^{\bullet}_{A}(x, x)
\]
such that
\[
\begin{gathered}
m_1(\varepsilon)=0 \\
m_2(\varepsilon \otimes Id)=m_2(Id \otimes \varepsilon)=Id \\
m_k(Id^{\otimes i}\otimes \varepsilon \otimes Id^{\otimes k-i-1})=0
\end{gathered} 
\]
for $0\le i\le k-1$.
\end{defi}
\medskip
\begin{defi}
Given $\Ainf$-categories $A,B$, a unital $\Ainf$-functor $f:A\to B$ \item is the data of map of sets $f:Ob(\mathcal{A})\to Ob(\mathcal{B})$, graded maps of degree $1-i$
\medskip
\[
f_i: Hom^{\bullet}_{A}(x_{i-1}, x_i)\otimes \dots \otimes Hom^{\bullet}_{A}(x_0, x_1)\to Hom^{\bullet}_{B}(f(x_0),f(x_i))
\]
$n\ge 1$, satisfying the system of equations
\medskip
\[
\sum_{n=r+t+s} (-1)^{sr+t} f_{r+t+1}(Id^{\otimes^r}\otimes m_s \otimes Id^{\otimes^t})=\sum_{\substack{1\le r\le n\\ i_1+\dots +i_r=n}} (-1)^{\epsilon_r} m'_{r}(f_{i_1}\otimes \dots \otimes f_{i_r})
\]
where
\medskip
\[
\epsilon_r=\epsilon_r(i_1,\dots ,i_r)=\sum_{2\le k\le r}\left( (1-i_k)\sum_{1\le l\le k-1} i_l \right)
\] 
and 
\medskip
\[
\begin{gathered}
f_1(1_x)=1_{f(x)} \\
f_n(a_1\otimes \dots \otimes a_{j-1}\otimes 1_x\otimes a_{j+1}\otimes \dots a_n)=0
\end{gathered} 
\]
for $n>1$, $1< j< n$.
\end{defi}
\medskip
\begin{remark}
As in the case of dg-categories, unital $\Ainf$-categories with a set of objects $S$ are identified with unital $\Ainf$-algebra objects in $Qu(S,Vect_{\mathbb{Z}}(\mathbb{K}))$. A non-unital $\Ainf$-category over a set of objects $S=Ob(A)$ is a non-unital $\Ainf$-algebra object in $Qu(S,Vect_{\mathbb{Z}}(\mathbb{K}))$. We denote by $\Ainf Cat_{nu}$ the category of non-unital $\Ainf$-categories.
\end{remark}
\medskip
\begin{defi}
A differential graded cocategory is given by a a set of objects $Ob(C)$, a graded quiver over it $Hom^{\bullet}_{C}=\lbrace Hom^{\bullet}_{C}(x,y) \rbrace_{x,y\in Ob(C)}$ together with a map of graded quivers of degree $+1$
\medskip
\[
b:C\to C
\]
and a map of graded quivers of degree $0$
\medskip
\[
\Delta:C\to C\otimes C 
\]
such that
\medskip
\[
(b\otimes Id_C+Id_C\otimes b)\circ \Delta= \Delta\circ b
\]
A graded cocategory is counital if endowed with a degree $0$ morphisms of graded quivers
\medskip
\[
\eta:C\to \mathbb{K}
\]
such that
\medskip
\[
(Id_C\otimes \eta) \circ \Delta = (\eta\otimes Id_C)\circ \Delta= Id_C
\]
\end{defi}
\medskip
\begin{defi}
Given an graded quiver $V$ over a set $S$, its bar construction is the graded quiver
\medskip
\[
B(V)=\bigoplus_{n\ge 1} (V[1])^{\otimes n} 
\]
together with the degree $0$ morphism
\medskip
\[
\Delta: B(V)\to B(V)\otimes B(V) 
\]
given by separation of tensors
\medskip
\[
\Delta(v_1\otimes \cdots \otimes v_n)=\sum_{1\le i \le n} (v_1\otimes \cdots \otimes v_i)\otimes (v_{i-1}\otimes \cdots \otimes v_n)
\]
which defines on $B(V)$ a structure of a graded coalgebra (non-counital) object in $Qu(S,Vect_{\mathbb{Z}}(\mathbb{K}))$.
\end{defi}
\medskip
\begin{proposition}
Given a graded quiver $A$ on a set $Ob(A)$, there exists a bijection between (non-unital) $\Ainf$-category structures on $A$ and differentials $b$ on $B(A)$ making $(B(A),b,\Delta)$ a differential graded cocategory. Moreover, the bar construction extends to a functor
\medskip
\[
B:\Ainf Cat_{nu}\to dgCoCat
\]
whose essential image is given by dg-cocategories which are cocomplete \cite{LH}.
\end{proposition}
\medskip
\begin{defi}
Given a graded quiver $V$ over a set $S$, its cobar construction is the graded quiver
\medskip
\[
\Omega(V)\bigoplus_{n> 1} (V[-1])^{\otimes n} 
\]
together with the degree $0$ morphism
\medskip
\[
\mu: \Omega(V)\otimes \Omega(V)\to \Omega(V)
\]
given by tensor multiplication
\medskip
\[
\mu((v_1\otimes \cdots \otimes v_i)\otimes (v_{i-1}\otimes \cdots \otimes v_n))=(v_1\otimes \cdots \otimes v_i\otimes v_{i-1}\otimes \cdots \otimes v_n)
\]
which defines on $ \Omega(V)$ a structure of graded algebra object in $Qu(S,Vect_{\mathbb{Z}}(\mathbb{K}))$.
\end{defi}
\medskip
\begin{proposition}
The cobar construction extends to a functor
\medskip
\[
\Omega: dgCoCat_{nco}\to dgCat_{nu}
\]
which is the left adjoint of the restriction of the bar construction to $dgCat_{nu}$.
\end{proposition}

\subsection{Augmentation, reduction and enveloping dg-category.}		

\begin{defi}
Given a (unital) dg-category $D$, its reduction is the non-unital dg-category 
\medskip
\[
\overline{D}=coKer(\varepsilon_D: \mathbb{K}\to D)
\]
where the cokernel is taken in the category of graded quivers.
\end{defi}
\medskip
\begin{defi}
Given a non-unital dg-category $E$, its augmentation is the unital dg-category
\medskip
\[
E^+=E\oplus \mathbb{K} 
\] 
with the unique dg-category structure making the inclusion $\mathbb{K}\to E\oplus \mathbb{K}$ the unit.
\end{defi}
\medskip
\begin{defi}
Given a (unital) $\Ainf$-category $A$ its reduction is the non-unital $\Ainf$-category 
\medskip
\[
\overline{A}=coKer(\varepsilon_A: \mathbb{K}\to A)
\]
where the cokernel is taken in the category of graded quivers.
\end{defi}
\medskip
\begin{defi}
Given a non-unital $\Ainf$-category $F$, its augmentation is the unital $\Ainf$-category
\medskip
\[
F^+=F\oplus \mathbb{K} 
\] 
with the unique $\Ainf$-category structure making the inclusion $\mathbb{K}\to F\oplus \mathbb{K}$ the unit.
\end{defi}
\medskip
\begin{defi}
Given a (counital) dg-cocategory $C$ its reduction is the non-counital dg-category 
\medskip
\[
\overline{C}=Ker(\eta_C: \mathbb{K}\to C)
\]
where the kernel is taken in the category of graded quivers.
\end{defi}
\medskip
\begin{defi}
Given a non-counital dg-category $B$, its augmentation is the counital dg-cocategory
\medskip
\[
B^+=B\oplus \mathbb{K} 
\] 
with the unique dg-cocategory structure making the projection $B\oplus \mathbb{K}\to \mathbb{K}$ the counit.
\end{defi}
\medskip
\begin{lemma}
Given dg-categories $C,D$, there exists natural bijections
\medskip
\[
Hom_{dgCat}(C,D)\simeq Hom_{dgCat_{nu}}(\overline{C},\overline{D})
\]
and for non-unital dg-categories $C',D'$ there exists natural bijections
\medskip
\[
Hom_{dgCat_{nu}}(C',D')\simeq Hom_{dgCat}((C')^+,(D')^+)
\]
\end{lemma}
\medskip
\begin{lemma}
Given $\Ainf$-categories $A,B$, there exists natural bijections
\medskip
\[
Hom_{\Ainf Cat}(A,B)\simeq Hom_{\Ainf Cat_{nu}}(\overline{A},\overline{B})
\]
and for non-unital $\Ainf$-categories $A',B'$ there exists natural bijections
\medskip
\[
Hom_{\Ainf Cat_{nu}}(A',B')\simeq Hom_{\Ainf Cat}((A')^+,(B')^+)
\]
\end{lemma}
\medskip
\begin{lemma}
Given dg-cocategories $E,F$, there exists natural bijections
\medskip
\[
Hom_{dgCoCat}(E,F)\simeq Hom_{dgCoCat_{ncu}}(\overline{E},\overline{F})
\]
and for non-counital dg-cocategories $E',F'$ there exists natural bijections
\medskip
\[
Hom_{dgCoCat_{ncu}}(E',F')\simeq Hom_{dgCoCat}((E')^+,(F')^+)
\]
\end{lemma}
\medskip
\begin{defi}
The augmented bar construction is the functor
\medskip
\[
B^+: \Ainf Cat\to dgCoCat
\]
defined by
\medskip
\[
B^+(A)=(B(\overline{A}))^+
\]
The augmented cobar construction is the functor
\medskip
\[
\Omega^+: dgCoCat\to dgCat
\]
defined by
\medskip
\[
\Omega^+(C)=(\Omega (\overline{C}))^+
\]
\end{defi}
\medskip
\begin{defi}
Given an $\Ainf$-category $A$, its enveloping dg-category is the dg-category
\medskip
\[
U(A)=(\Omega B(\overline{A}))^+
\]
\end{defi}
\medskip
\begin{remark}
Given $\Ainf$-categories $A,B$, there exists natural bijections of sets
\medskip
\[
Hom_{\Ainf Cat}(A,B)\simeq Hom_{dgCoCat}(B^+(A),B^+(B))
\]
This bijection is given by composing the following chain of bijections
\medskip
\[
\begin{gathered}
Hom_{\Ainf Cat}(A,B)\simeq Hom_{\Ainf Cat_{nu}}(\overline{A},\overline{B})\simeq Hom_{dgCoCat_{ncu}}(B(\overline{A}),B(\overline{B}))\simeq \\
\simeq Hom_{dgCoCat}(B(\overline{A})^+,B(\overline{B})^+)\simeq Hom_{dgCoCat}(B^+(A),B^+(B))
\end{gathered} 
\]
\end{remark}
\medskip
\begin{remark}
The construction of the enveloping dg-category defines an adjunction
\medskip
\[
\adj{U}{\Ainf Cat}{dgCat}{i}
\]
where $i$ is the inclusion of $dgCat$ in $\Ainf Cat$. Indeed, there exists natural bijections of sets
\medskip
\[
\begin{gathered}
Hom_{dgCat}(U(A),D)\simeq Hom_{dgCat}(\Omega(B(\overline{A}))^+,D)\simeq Hom_{dgCat_{nu}}(\Omega(B(\overline{A})),\overline{D})\simeq \\
\simeq Hom_{dgCoCat_{ncu}}(B(\overline{A}),B(\overline{D}))\simeq Hom_{\Ainf Cat}(A,D)\simeq Hom_{\Ainf Cat}(A,i(D))
\end{gathered} 
\]
In particular, for $A=D$, we get a natural morphism of dg-categories 
\medskip
\[
\gamma_D:U(D)\to D
\]
corresponding to the $\Ainf$-morphism $Id_{D}$. The enveloping dg-category of a dg-category $D$ is a cofibrant dg-category in the Tabuada model structure, being the free tensor dg-category over a given graded quiver. Moreover, the morphism $\gamma_D$ is an equivalence of dg-categories.
\end{remark}

\newpage

\section{dg-bimodules and $\Ainf$-bimodules.}\label{APPB}

\subsection{Differential graded modules and bimodules.}
Given a dg-category $D$, a (unital) dg-module over $D$ is a dg-functor
\[
M: D\to Ch(\mathbb{K})
\]
Explicitly, it is given by a chain complex $M(y)$, for every $y\in Ob(D)$, and maps of degree $0$
\medskip
\[
\sigma(y_0,y_1): M(y_0) \otimes Hom^{\bullet}_D(y_0,y_1) \to M(y_1)
\]
such that 
\medskip
\[ 
\left \{
  \begin{tabular}{ccc}
  $d_{M}\circ \sigma= \sigma\circ (Id_M\otimes d_{D}+d_M\otimes Id_D)$ \\
  \\
  $\sigma(Id_M\otimes \varepsilon_D)=Id_M$
  \end{tabular}
  \right.
\]
Given dg-modules $M_0,M_1$ over a dg-category $D$, a morphism of dg-modules is a morphism of dg-functors. We denote by $Mod(D)$ the category whose objects are dg-modules over $D$ and morphisms are morphisms of dg-modules.
\medskip
\begin{example}
Given a dg-category $D$ and an object $y\in D$, there is a $D^{op}$-dg-module, called the representable dg-module associated to $y$, $h_y$ that, to an object $y_0$, associates the complex 
\medskip
\[
h^{dg}(y)(y_0)=Hom^{\bullet}_{D}(y_0,y)
\]
with differential induced by the differential of $D$ and dg-action induced by the composition in $D$.
\end{example}
\medskip
\begin{defi}
Given a dg-category $D$, the dg-category of dg-modules over it, $Mod^{\bullet}(D)$, is the dg-category whose objects are dg-modules over $D$ and morphisms of degree $d$ between two given dg-modules $r\in Hom^d_{Mod^{\bullet}(D)}(M_0,M_1)$ are given by a family of maps of degree $d$
\medskip
\[
r(y):M_0(y)\to M_1(y)
\]
such that
\medskip
\[
r\circ \sigma_{M_0}=\sigma_{M_1}\circ (Id_D\otimes r)
\]
The differential
\medskip
\[
d: Hom^d_{Mod^{\bullet}(D)}(M_0,M_1)\to Hom^{d+1}_{Mod^{\bullet}(D)}(M_0,M_1)
\]
is given by
\medskip
\[
d(r)=d_{M_1}\circ r - (-1)^{deg(r)} r\circ d_{M_0}
\]
\end{defi}
\medskip
\begin{remark}
There exists an isomorphism of dg-categories
\medskip
\[
Mod^{\bullet}(D)\xrightarrow{\simeq} dgFunc^{\bullet}(D, Ch(\mathbb{K}))
\]
and the category $Mod(D)$ is identified with $Z^0(Mod^{\bullet}(D))$. Moreover, there exists a model structure on $Mod(D)$ \cite{To} for which equivalences and fibrations are defined object-wise. Such model structure, together with the enrichment given by the dg-category $Mod^{\bullet}(D)$, makes $Mod(D)$ into a $Ch(\mathbb{K})$-enriched model category. In particular we have natural equivalences of categories
\medskip
\[
H^0(Int(Mod^{\bullet}(D)))\xrightarrow{\sim} Ho(Mod(D))
\]
where $Int(Mod^{\bullet}(D))$ is the full dg-subcategory if $Mod^{\bullet}(D)$ whose objects are fibrant and cofibrant dg-bimodules. Moreover, there exists a dg-functor, called the dg-Yoneda embedding
\medskip
\[
h^{dg}:D\to Mod^{\bullet}(D^{op})
\]
associating to each object of $D$ its representable dg-module. This dg-functor is fully-faithful. 
\end{remark}
\medskip
Given dg-categories $D,E$, a (unital) dg-bimodule $M$ over $D$ and $E$ is a dg-functor
\[
M: D\otimes E^{op}\to Ch(\mathbb{K})
\]
Explicitly, it is given by a chain complex $M(y,z)$, for every $y\in Ob(D)$ and $z\in Ob(E)$, and maps of degree $0$
\medskip
\[
\sigma(y_0,y_1,z_0,z_1): Hom^{\bullet}_E(z_1,z_0)\otimes M(y_0,z_0) \otimes Hom^{\bullet}_D(y_0,y_1) \to M(y_1,z_1)
\]
such that 
\medskip
\[ 
\left \{
  \begin{tabular}{ccc}
  $d_{M}\circ \sigma= \sigma\circ (d_D\otimes Id_M\otimes Id_E+Id_D\otimes d_M\otimes Id_E+Id_D\otimes Id_M\otimes d_E)$ \\
  \\
  $\sigma(\varepsilon_D \otimes Id_M\otimes \varepsilon_E)=Id_M$
  \end{tabular}
  \right.
\]
Given dg-bimodules $M_0,M_1$ over a dg-categories $D$ and $E$, a morphism of dg-bimodules is a morphism of dg-functors. We denote by $Mod(D,E)$ the category whose objects are dg-bimodules over $D$ and $E$ and morphisms are morphisms of dg-bimodules. There exists a canonical identification $Mod(D,E)\simeq Mod(D\otimes E^{op})$.
\medskip
\begin{defi}
Given dg-categories $D,E$, the dg-category of dg-bimodules is the dg-category $Mod^{\bullet}(D,E)$ defined by
\medskip
\[
Mod^{\bullet}(D,E)=dgFunc^{\bullet}(D\otimes E^{op}, Ch(\mathbb{K}))
\]
\end{defi}
\medskip
\begin{remark}
The category $Mod(D,E)$ is identified with $Z^0(Mod^{\bullet}(D,E))$. Moreover, there is a model structure on $Mod(D,E)$ \cite{To} for which equivalences and fibrations are defined object-wise. Such model structure, together with the enrichment given by the dg-category $Mod^{\bullet}(D,E)$, makes $Mod(D,E)$ into a $Ch(\mathbb{K})$-enriched model category. In particular we have natural equivalences
\medskip
\[
H^0(Int(Mod^{\bullet}(D,E)))\simeq Ho(Mod(D,E))
\]
where $Int(Mod^{\bullet}(D,E))$ is the full dg-subcategory if $Mod^{\bullet}(D,E)$ whose objects are fibrant and cofibrant dg-bimodules.
\end{remark}
\medskip
\begin{defi}
A dg-bimodule $M\in Mod(D,E)$ is called right quasi-representable if, for every $y\in Ob(D)$, the dg-module $M(y,-)\in Mod(E^{op})$ is weakly-equivalent to the representable dg-module $h_{z(y)}$, for some $z(y)\in Ob(E)$.
\end{defi}

\subsection{$\Ainf$-modules and bimodules. Differential graded comodules and bicomodules}
\begin{defi}
Given a (unital) $\Ainf$-category, a (untial) $\Ainf$-module over it is given by a graded quiver $M$ over $Ob(A)$ together with graded maps of degree $2-i$
\medskip
\[
m^M_i:M(y_0)\otimes Hom^{\bullet}_{A}(y_0, y_1)\otimes \dots \otimes Hom^{\bullet}_{A}(y_{i-2}, y_{i-1})\to M(y_{i-1})
\]
$i\ge 1$, such that
\medskip
\begin{equation}
\sum_{n=i+j+k} (-1)^{ik+j} m_{i+j+1}(Id^{\otimes^i}\otimes m_k \otimes Id^{\otimes^j})=0
\end{equation}
for $n\ge 1$ and where the $m_i$'s are the one defining the action on $M$ or the one given by the $\Ainf$ structure on $A$ depending on the obvious compositions compatibilities. Moreover, they satisfy the unitality conditions:
\medskip
\[
\begin{gathered}
m^M_2(Id_M\otimes \varepsilon)=Id_M \\
m^M_k(Id_M^{\otimes i}\otimes \varepsilon \otimes Id_M^{\otimes k-i-1})=0
\end{gathered} 
\]
for $0\le i\le k-1$, $k\ge 3$.
\end{defi}
\medskip
\begin{defi}
Given $\Ainf$-modules $M_0,M_1$ over an $\Ainf$-category $A$, a morphism of $\Ainf$-modules is given by a family of morphisms of degree $1-i$
\medskip
\[
f_i:M_0(y_0)\otimes Hom^{\bullet}_{A}(y_0, y_1)\otimes \dots \otimes Hom^{\bullet}_{A}(y_{i-2}, y_{i-1})\to M_1(y_{i-1})
\]
$i\ge 1$, such that
\medskip
\[
\begin{gathered}
\sum_{n=r+t+s} (-1)^{sr+t} f_{r+t+1}(Id^{\otimes^r}\otimes m^{M_0}_s \otimes Id^{\otimes^t})=\sum_{\substack{1\le r\le n\\ i_1+\dots +i_r=n}} m^{M_1}_{s+1}(f_r\otimes Id^{\otimes s} \\
f_n(Id_M \otimes \dots \otimes Id_A \otimes \varepsilon \otimes Id_A\otimes \dots Id_A)=0
\end{gathered} 
\]
for $n>1$, $1< j< n$.
\end{defi}
\medskip
\begin{defi}
Given and $\Ainf$-category, the category $Mod_{\infty}(A)$ is the category whose objects are $\Ainf$-modules over $A$ and morphisms are morphisms of $\Ainf$-modules.  
\end{defi}
\medskip
\begin{defi}
Given a coaugmented dg-cocategory $C$, a (counital) dg-comodule over $C$ is given by a graded quiver $N$ over $Ob(C)$, together with maps of degree $+1$
\medskip
\[
b_{N(y)}:N(y)\to N(y)
\]
for every $y\in Ob(C)$, and maps of degree $0$, called coaction maps,
\medskip
\[
\Delta_{N}(y_0,y_1): N(y_1)\to  N(y_0)\otimes Hom^{\bullet}_C(y_0,y_1) 
\]
such that 
\medskip
\[
\begin{gathered}
b_{M}^2=0 \\
(Id_N\otimes \Delta_C)\circ \Delta_N=(\Delta_N\otimes Id_N)\circ \Delta_N
\end{gathered} 
\]
and
\medskip
\[
\Delta_N\circ (Id_N\otimes \eta_C)=Id_N
\]
\end{defi}
\medskip
\begin{defi}
Given a coaugmented dg-cocategory $C$, the dg-category $CoMod^{\bullet}(C)$ is the dg-category whose objects are counital cocomplete dg-comodule \cite{LH} and a morphism $F$ of degree $d$ between given comodules $N_0,N_1$, is a map of graded quivers of degree $d$
\medskip
\[
F:N_0\to N_1
\]
which is compatible with the comodule structures and couints in the obvious way. The differential
\medskip
\[
d:Hom^d_{CoMod^{\bullet}(C)}(N_0,N_1)\to Hom^d_{CoMod^{\bullet}(C)}(N_0,N_1)
\]
is given by the commutator
\medskip
\[
d(F)=b_{N_1}\circ F - (-1)^d F\circ b_{N_0}
\]
\end{defi}
\medskip
\begin{defi}
Given a coaugmented dg-cocategory $C$, the category $CoMod(C)$ is the category $Z^0(CoMod^{\bullet}(C))$.
\end{defi}
\medskip
\begin{remark}
Given $A$ an $\Ainf$-category, there is a notion of weak-equivalences  (or better, a model category without limits) in the category $Mod_{\infty}(A)$ and a notion of homotopy for morphisms of $\Ainf$-modules. There are, moreover, natural equivalences
\medskip
\[
Mod_{\infty}(A)[W^{-1}]\simeq \bigslant{Mod_{\infty}(A)}{\sim}
\]
where $W$ is the class of weak-equivalences  in $Mod_{\infty}(A)$ and $\sim$ is the relation of homotopy in $Mod_{\infty}(A)$. Given a coaugmented dg-cocategory $C$, there is a model structure on $CoMod(C)$. Those notions are compatible, in the sense that there exists natural functors 
\medskip
\[
B_A: Mod_{\infty}(A)\to CoMod(B^+(A))
\]
inducing equivalences in the localizations.
\end{remark}
\medskip
\begin{defi}
Given $A$ and $\Ainf$-category, the dg-category $\mathcal{C}_{\infty}(A)$ is the dg-category whose objects are $\Ainf$-modules over $A$ and complex of morphism
\medskip
\[
Hom^{\bullet}_{\mathcal{C}_{\infty}(A)}(M_0,M_1)=Hom^{\bullet}_{CoMod^{\bullet}(B^+(A))}(B_A(M_0),B_A(M_1))
\]
\end{defi}
\medskip
\begin{remark}
By the definition of $\mathcal{C}_{\infty}(A)$, we have a dg-functor
\medskip
\[
B_A: \mathcal{C}_{\infty}(A)\to CoMod^{\bullet}(B^+(A))
\]
which is an isomorphims of dg-categories. Moreover, the dg-category $\mathcal{C}_{\infty}(A)$ computes the localization of $Mod_{\infty}(A)$ at the class of weak-equivalences, in the sense that there exists natural equivalences of categories
\medskip
\[
Mod_{\infty}(A)[W^{-1}]\xrightarrow{\sim} \bigslant{Mod_{\infty}(A)}{\sim}\xrightarrow{\sim} H^0(\mathcal{C}_{\infty}(A))
\]

Given $A$ an $\Ainf$-category, there exists an $\Ainf$-functor, called the $\Ainf$-Yoneda embedding
\medskip
\[
h^{\infty}:A\to \mathcal{C}_{\infty}(A)
\]
whose image coincides with the dg-Yoneda embedding, if $A$ is a dg-category.
\end{remark}

\subsection{Comparison between dg-bimodules and $\Ainf$-bimodules}
Most of the results that follow are taken from \cite{LH}. We need to remark that those results are stated, in their original form, for a slightly different model of the enveloping dg-category, which is given in the context of augmented dg-categories. However, those results do not depend really on the fact that the dg-categories considered are augmented and on the specific model used for the enveloping dg-category. The same results hold, slightly modifying the constructions and the proofs, with the model of the enveloping dg-category used in this paper.  

Given $A$ an $\Ainf$-category, there is a natural functor
\medskip
\[
R_{\tau}(A):Mod(U(A))\to CoMod(B^+(A))
\]
which is the right adjoint of a Quillen equivalence of model categories. Such $R_{\tau}(A)$ is defined via an acyclic twisted cochain \cite{LH}. Moreover, there is a functor
\medskip
\[
J_{A}:Mod(U(A))\to Mod_{\infty}(A)
\]
which induces equivalences in the localizations. Those functors are compatible with $B_A$, in the sense that we have a commutative diagram
\medskip
\[
 \begin{tikzpicture}
    \def\x{1.5}
    \def\y{-1.2}
    \node (A2_2) at (4*\x, 2.5*\y) {$Mod_{\infty}(A)$};
    \node (A2_1) at (2*\x, 1*\y) {$Mod(U(A))$};
    \node (A1_2) at (0*\x, 2.5*\y) {$CoMod(B^+(A))$};
   
     \path (A2_1) edge [->] node [auto,swap] {$\scriptstyle{J_{A}}$} (A2_2);
     \path (A2_1) edge [->] node [auto,swap] {$\scriptstyle{R_{\tau}(A)}$} (A1_2);
   \path (A2_2) edge [->] node [auto,swap] {$\scriptstyle{B_{A}}$} (A1_2);
          \end{tikzpicture}
\]

Analogous constructions and statements hold for the case of $\Ainf$-bimodules and dg-cobimodules. Given $\Ainf$-categories $A$ and $B$, $\mathcal{A}_{\infty}$-bimodules are the objects of category $Mod_{\infty}(A,B)$ where morphisms are maps of $\mathcal{A}_{\infty}$-bimodules \cite{LH}. Such category comes equipped with a notion of weak-equivalences , namely $\Ainf$ quasi-isomorphism of $\Ainf$-bimodules, and a notion of homotopy between morphisms of $\Ainf$-bimodules which are compatible, in the sense that there exist natural equivalences of categories
\medskip
\[
Mod_{\infty}(A,B)[W^{-1}]\simeq \bigslant{Mod_{\infty}(A,B)}{\sim}
\]
where $\sim$ is the equivalence relation induced by the notion of homotopy and $W$ is the class of $\Ainf$-quasi isomorphisms. Moreover, there is a dg-enrichment of the category $Mod_{\infty}(A,B)$, denoted by $\Cinf(A,B)$ \cite{LH} which is compatible with localization at $\Ainf$ quasi-isomorphims, giving natural equivalences of categories
\medskip
\[
Mod_{\infty}(A,B)[W^{-1}]\simeq \bigslant{Mod_{\infty}(A,B)}{\sim}\simeq H^0(\Cinf(A,B))
\]
In a similar way one can define the categories of counital cocomplete dg-bicomodules $CoMod(C,C')$ over coaugmented counital dg-cocategories $C,C'$. Such category is identified with
\medskip
\[
CoMod(C,C')\xrightarrow{\simeq} CoMod(C\otimes (C')^{op})
\]
where $C^{op}$ is the dg-cocategory with opposite cocomposition and $\otimes$ is the tensor product as dg-quivers. This identification provides $CoMod(C,C')$ with a model structure and a dg-enrichment $CoMod^{\bullet}(C,C')$ such that there exists an isomorphism of dg-categories
\medskip
\[
CoMod^{\bullet}(C,C')\xrightarrow{\simeq} CoMod^{\bullet}(C\otimes (C')^{op})
\]
In the case the dg-cocategories are the augmented bar construction of some dg-category, say $D$ and $E$, there exists a functor
\medskip
\[
B_{(D,E)}:Mod_{\infty}(D,E)\to CoMod(B^+(D),B^+(E))
\]
which induces an equivalence on the localizations
\medskip
\[
Mod_{\infty}(D,E)[W^{-1}]\xrightarrow{\sim} Ho(CoMod(B^+(D),B^+(E)))
\]
The same functor extends to an isomorphism of dg-categories
\medskip
\[
B(D,E): \Cinf(D,E)\xrightarrow{\simeq} CoMod(B^+(D),B^+(E))
\]
Also, there exists a natural functor
\medskip
\[
J_{(D,E)}: Mod(U(D),U(E))\to Mod_{\infty}(D,E)
\]
which induces natural equivalences on the localizations
\medskip
\[
Mod(U(D),U(E))[W^{-1}]\to Mod_{\infty}(D,E)[W^{-1}]
\]
In particular, for dg-categories $D$ and $E$, there exists a commutative diagram of natural functors 
\medskip
\[
 \begin{tikzpicture}
    \def\x{1.5}
    \def\y{-1.2}
    \node (A2_2) at (4*\x, 2.5*\y) {$Mod_{\infty}(D,E)$};
    \node (A2_1) at (2*\x, 1*\y) {$Mod(U(D),U(E))$};
    \node (A1_2) at (0*\x, 2.5*\y) {$CoMod(B^+(D), B^+(E))$};
   
     \path (A2_1) edge [->] node [auto] {$\scriptstyle{J(D,E)}$} (A2_2);
     \path (A2_1) edge [->] node [auto,swap] {$\scriptstyle{R_{\tau}(D,E)}$} (A1_2);
   \path (A2_2) edge [->] node [auto,swap] {$\scriptstyle{B(D,E)}$} (A1_2);
          \end{tikzpicture}
\]
where the functor $R_{\tau}(D,E)$ is defined via the identification
\medskip
\[
 \begin{tikzpicture}
    \def\x{1.5}
    \def\y{-1.2}
    \node (A2_2) at (4*\x, 2.5*\y) {$CoMod(B^+(D)\otimes B^+(E)^{op})$};
    \node (A2_1) at (4*\x, 1*\y) {$CoMod(B^+(D),B^+(E))$};
    \node (A1_2) at (0*\x, 2.5*\y) {$Mod(U(D)\otimes U(E)^{op})$};
    \node(A1_1) at (0*\x, 1*\y) {$Mod(U(D),U(E))$};
   
   \path (A1_1) edge [->] node [auto,swap] {$\scriptstyle{R_{\tau}(D,E)}$} (A2_1);
   \path (A1_1) edge [->] node [auto,swap] {$\scriptstyle{\simeq}$} (A1_2);
   \path (A1_2) edge [->] node [auto,swap] {$\scriptstyle{R_{\tau}}$} (A2_2);
    \path (A2_1) edge [->] node [auto,swap] {$\scriptstyle{\simeq}$} (A2_2);
          \end{tikzpicture}
\]
Here the functor $R_{\tau}$
\medskip
\[
R_{\tau}:Mod(U(D)\otimes U(E)^{op})\to CoMod(B^+(D)\otimes B^+(E)^{op})
\]
is defined via a twisting acyclic cochain as in \cite{LH}.

\raggedright
\newpage

\end{document}